\documentclass{article}
\usepackage{fancyhdr}
\usepackage{amscd}
\usepackage[dvipdfmx]{graphicx}
\usepackage{amsmath}
\usepackage{amssymb}
\usepackage{amsthm}
\usepackage{mathrsfs}
\newtheorem{dfn}{Definition}[section]
\newtheorem{thm}[dfn]{Theorem}
\newtheorem{prop}[dfn]{Proposition}
\newtheorem{lem}[dfn]{Lemma}

\newtheorem{rem}[dfn]{Remark}

\newtheorem{ass}[dfn]{Assumption}

\newtheorem{fact}[dfn]{Fact}

\numberwithin{equation}{section}

\begin{document}

\title{Geometric treatments and a common mechanism in finite-time singularities for autonomous ODEs}

\author{Kaname Matsue\thanks{Institute of Mathematics for Industry, Kyushu University, Fukuoka 819-0395, Japan {\tt kmatsue@imi.kyushu-u.ac.jp}} $^{,}$ \footnote{International Institute for Carbon-Neutral Energy Research (WPI-I$^2$CNER), Kyushu University, Fukuoka 819-0395, Japan}
}
\maketitle

\begin{abstract}
Geometric treatments of blow-up solutions for autonomous ordinary differential equations and their blow-up rates are concerned.
Our approach focuses on the type of invariant sets at infinity via compactifications of phase spaces, and dynamics on their center-stable manifolds. 
In particular, we show that dynamics on center-stable manifolds of invariant sets at infinity with appropriate time-scale desingularizations as well as blowing-up of singularities characterize dynamics of blow-up solutions as well as their rigorous blow-up rates.
\end{abstract}


\section{Introduction}
This paper aims at geometric treatments of blow-up solutions for differential equations as a sequel of the author's previous article \cite{Mat}.
In \cite{Mat}, the author characterizes blow-up solutions for a class of vector fields called {\em asymptotically quasi-homogeneous vector fields at infinity} in a simple case from the geometric viewpoint.
More precisely, {\em hyperbolic} equilibria and periodic orbits at infinity for such vector fields at infinity on compactified phase spaces with appropriate time-scale desingularizations induce blow-up solutions in terms of their stable manifolds.
In particular, trajectories on stable manifolds of such hyperbolic invariant sets for desingularized vector fields correspond to blow-up solutions for original vector fields. 
Moreover, their blow-up rates are uniquely characterized by the nonlinearity of vector fields.
The blow-up rates are often called of {\em type-I} in the field of (partial) differential equations.
These results answer, in a simple case, the fundamental question of blow-up problems for differential equations; {\em whether or not a solution blows up} and, if so, {\em when, where and how the solution blows up}.

\par
On the other hand, there are very rich variations of blow-up solutions and their behavior. 
For example, in the field of partial differential equations, it is well-known (e.g., \cite{FHV2000}) that the Cauchy problem for a semilinear heat equation
\begin{equation}
\label{heat-intro-PDE}
\begin{cases}
u_t = \Delta u + |u|^{p-1}u & \text{ in $\mathbb{R}^N\times (0,T)$,}\\
u(x,0) = u_0(x) \geq 0 & \text{ in $\mathbb{R}^N$}
\end{cases}
\end{equation}
with $T>0$ and $u_0\in L^\infty(\mathbb{R}^N)$ possess a blow-up solution satisfying
\begin{equation}
\label{type-2}
\lim_{t\to \infty} (T-t)^{1/(p-1)}\|u(t,\cdot)\|_{\infty} = \infty
\end{equation}
under appropriate choice of $p$ and $N$.
Blow-up solutions for (\ref{heat-intro-PDE}) satisfying (\ref{type-2}) are often called {\em type-II blow-up solutions}.
The terminology \lq\lq type-II" are used in other systems such as the Keller-Segel system (e.g., \cite{M2016}).
Type-II blow-up solutions, and possibly other type blow-ups, are understood in terms of their blow-up rates, which can contain richer information than nonlinearity of vector fields, as shown in the above example.
One of our aims here is to understand such blow-up rates from viewpoints of geometry and dynamical systems.
As mentioned above, the author's preceding work shows that hyperbolic invariant sets at infinity characterize type-I blow-up solutions, which indicates that a key point for characterizing blow-up solutions with blow-up rates other than type-I will be {\em non-hyperbolic invariant sets at infinity}.
In general non-hyperbolic invariant sets has a variety of structures depending on dynamical systems and hence concrete analysis can be done individually.
Nevertheless, we expect that explicit asymptotic behavior of trajectories on {\em center manifolds of invariant sets at infinity}\footnote
{
Precise meanings are shown in successive sections.
} will provide divergent or blow-up solutions with precise asymptotic behavior such as blow-up rates as the preceding study \cite{Mat}.
\par
\bigskip
Throughout successive sections, consider the (autonomous) vector field
\begin{equation}
\label{ODE-original}
y' = f(y),
\end{equation}
where $f : U \to \mathbb{R}^n$ be a smooth asymptotically quasi-homogeneous function with an admissible domain $U\subset \mathbb{R}^n$ with respect to type $\alpha\in \mathbb{N}^n$ and order $k+1 > 1$ (Definitions \ref{dfn-index} and \ref{dfn-AQH}).
Our basic approach is briefly written as follows.
\begin{enumerate}
\item Apply {\em compactifications} of phase spaces and derivation of {\em desingularized vector fields} to (\ref{ODE-original}).
The infinity then corresponds to a hypersurface $\mathcal{E}$ called {\em horizon}.
\item Specify an invariant set $S$ on $\mathcal{E}$ for desingularized vector fields.
\item Solve explicit solutions which converge to $S$.
\item Transform the calculated solutions to those for the original system (\ref{ODE-original}).
\end{enumerate}
Our main result is the following : {\em if $S$ is non-hyperbolic for desingularized vector field in the sense stated above and we solve trajectories on $W^c(S)$, then they correspond to divergent or blow-up solutions with blow-up rates which are generally different from type-I}.
We can also say that the similar feature is revealed to other finite-time singularities such as {\em finite-time extinction}, {\em compacton traveling waves} or {\em quenching}.
\par
\bigskip
The rest of the paper is organized as follows.
In Section \ref{section-type-I}, we review the approach and result shown in \cite{Mat, MT2017} about characterization of blow-up solutions in terms of trajectories on stable manifolds of invariant sets at infinity for appropriately associated vector fields.
These results gives characterization of {\em type-I blow-ups}.
In Section \ref{section-mechanism}, we discuss a methodology how the asymptotic behavior of blow-up solutions with different blow-up rates from type-I can be derived.
The main issue there is a treatment of {\em non-hyperbolic} invariant sets at infinity.
We thus review the type of equilibria, which is often discussed in singularities of (planar) vector fields (e.g., \cite{DLA2006}).
Several equivalences among dynamical systems are reviewed, which help us with treating asymptotic behavior of trajectories simply.
Independently, we discuss the treatment for detecting other types of finite-time singularities such as {\em finite-time extinctions} and {\em compactons}.
As indicated in preceding studies in partial differential equations (e.g., \cite{K1992}), there are several common aspects among finite-time singularities including blow-ups, extinctions, compactons and quenching.
Our present treatments reveals such common mechanisms of singularities from the viewpoint of dynamical systems.
In Section \ref{section-example}, we provide various examples with calculating rigorous rates of finite-time singularities including blow-ups, extinctions, compactons, quenching and periodic blow-ups.
We see that the asymptotic rates can be derived component-wise whether or not it is faster than that associated with nonlinearity of vector fields; namely type-I rates.

\section{Compactifications and type-I blow-ups}
\label{section-type-I}

First we collect our present settings and results for detecting blow-up solutions with blow-up rates.
These notions and statements are based on \cite{Mat}.

\begin{dfn}[Homogeneity index and admissible domain]\rm
\label{dfn-index}
Let $\alpha = (\alpha_1,\cdots, \alpha_n)$ be a set of nonnegative integers.
Let the index set $I_\alpha$ as $I_\alpha=\{i\in \{1,\cdots, n\}\mid \alpha_i > 0\}$, which we shall call {\em the set of homogeneity indices associated with $\alpha = (\alpha_1,\cdots, \alpha_n)$}.
Let $U\subset \mathbb{R}^n$.
We say the domain $U\subset \mathbb{R}^n$ {\em admissible with respect to the sequence $\alpha$}
if
\begin{equation*}
U = \{x=(x_1,\cdots, x_n)\in \mathbb{R}^n \mid x_i\in \mathbb{R}\text{ if }i\in I_\alpha,\ (x_{j_1},\cdots, x_{j_{n-l}}) \in \tilde U\},
\end{equation*} 
where $\{j_1, \cdots, j_{n-l}\} = \{1,\cdots, n\}\setminus I_\alpha$ and $\tilde U$ is an $(n-l)$-dimensional open set.
\end{dfn}

\begin{dfn}[Asymptotically quasi-homogeneous vector fields]\rm
\label{dfn-AQH}
Let $f = (f_1,\cdots, f_n):U \to \mathbb{R}^n$ be a smooth function with an admissible domain $U\subset \mathbb{R}^n$ with respect to $\alpha$ such that $f$ is uniformly bounded for each $x_i$ with $i\in I_\alpha$, where $I_\alpha$ is the set of homogeneity indices associated with $\alpha$.
We say that $X = \sum_{j=1}^n f_j(x)\frac{\partial }{\partial x_j}$, or simply $f$ is an {\em asymptotically quasi-homogeneous vector field of type $(\alpha_1,\cdots, \alpha_n)$ and order $k+1$ at infinity} if
\begin{equation*}
\lim_{R\to +\infty} R^{-(k+\alpha_j)}\left\{ f_j(R^{\alpha_1}x_1, \cdots, R^{\alpha_n}x_n) - R^{k+\alpha_j}(f_{\alpha,k})_j(x_1, \cdots, x_n) \right\}
 = 0
 \end{equation*}
holds uniformly for $(x_1,\cdots, x_n)\in U_1$, where $f_{\alpha,k} = ((f_{\alpha,k})_1,\cdots, (f_{\alpha,k})_n)$ is a quasi-homogeneous vector field of type $(\alpha_1,\cdots, \alpha_n)$ and order $k+1$, and
\begin{equation*}
U_1 = \{x=(x_1,\cdots, x_n)\in \mathbb{R}^n \mid (x_{i_1}, \cdots, x_{i_l})\in S^{l-1},\ (x_{j_1},\cdots, x_{j_{n-l}}) \in \tilde U\},
\end{equation*}
where $\{i_1,\cdots, i_l\} = I_\alpha$.
\end{dfn}

Throughout the rest of this section, we consider (\ref{ODE-original}) such that $f$ is asymptotically quasi-homogeneous of type $\alpha\in \mathbb{N}^n_0$, where $\mathbb{N} = \mathbb{N}\cup \{0\}$, and order $k+1$ with $k\geq 1$.

\subsection{Compactifications associated with given types}

Quasi-homogeneous compactifications in our present arguments are restricted to the following.
\begin{dfn}[Directional compactifications]\rm
An (orthogonal) {\em directional compactification of type $\alpha$} is defined as
\begin{equation}
\label{cpt-dir-def}
y_i = \frac{1}{r^{\alpha_i}},\quad y_j = \frac{x_j}{r^{\alpha_j}}\ (j\not = i).
\end{equation}
\par
If our system is two dimensional and the type of $f$ is $(1, m)$ with $m \geq 1$, then {\em quasi-polar compactification} of type $(1, m)$ is defined as
\begin{equation}
\label{cpt-qpolar-def}
y_1 = \frac{{\rm Cs}\theta}{r},\quad y_2 = \frac{{\rm Sn}\theta}{r^m}
\end{equation}
A general form is stated in \cite{Mat}.
\end{dfn}
In this terminology \lq\lq the infinity" corresponds to a hypersurface $\mathcal{E}=\{r=0\}$, which we shall call {\em the horizon}.
Although we only use these compactifications in our arguments below, the essential idea can be applied with other type of compactifications (cf. \cite{Mat, MT2017}).
\par
If we apply the compactification (\ref{cpt-dir-def}) and time-scale desingularization
\begin{equation}
\frac{d\tau}{dt} = r(t)^{-k}
\end{equation}
to (\ref{ODE-original}), we obtain the following {\em desingularized vector field}
\begin{equation}
\label{ODE-desing-dir}
\begin{pmatrix}
\frac{dr}{d\tau_d} \\ \frac{dx_2}{d\tau_d} \\ \vdots \\ \frac{dx_n}{d\tau_d}
\end{pmatrix}
=
\begin{pmatrix}
-r & 0 & \cdots & 0 \\
0 & 1 & \cdots & 0\\
\vdots & \vdots & \ddots & \vdots \\
0 & 0 & \cdots & 1
\end{pmatrix}
\begin{pmatrix}
\alpha_1  & 0 & \cdots & 0 \\
\alpha_2 x_2 & 1 & \cdots & 0 \\
\vdots & \cdots & \cdots & \vdots\\
\alpha_n x_n & 0 & \cdots & 1
\end{pmatrix}^{-1}
\begin{pmatrix}
\hat f_1 \\ \hat f_2 \\ \vdots \\ \hat f_n
\end{pmatrix}\equiv g(r,x_2,\cdots, x_n),
\end{equation}
where 
\begin{equation*}
\hat f_j(r, x_2, \cdots, x_n) := r^{k+\alpha_j} f_j(r^{-\alpha_1}, r^{-\alpha_2}x_2, \cdots, r^{-\alpha_n}x_n),\quad j=1,\cdots, n
\end{equation*}
and we assume that, without the loss of generality, the index $i$ in (\ref{cpt-dir-def}) is $1$.
In particular, divergent solutions of (\ref{ODE-original}) corresponds to global trajectories for (\ref{ODE-desing-dir}) asymptotic to $\mathcal{E}$ (e.g., \cite{EG2006, Mat}).
Now we are ready to study blow-up solutions in terms of dynamics at infinity.
\begin{rem}[Remainder of Landau's symbols]
Let $a\in \mathbb{R}$ or $\{+\infty\}$.
For real-valued functions $f,g$, 
\begin{itemize}
\item \lq\lq $f(x) \in \Theta(g(x))$ as $x\to a-0$" means that, there exist $\delta$ and positive constants $k_1, k_2$ such that 
\begin{equation*}
\forall x<a \text{ with }|x-a| < \delta \quad \Rightarrow \quad k_1 g(x) \leq f(x) \leq k_2 g(x).
\end{equation*}
\item \lq\lq $f(x) = o(g(x))$ as $x\to a$" means 
\begin{equation*}
\lim_{x\to a}\left| \frac{f(x)}{g(x)}\right| = 0.
\end{equation*}
\item \lq\lq $f(x) \sim g(x)$ as $x\to a$" means that $f(x)-g(x) = o(g(x))$ as $x\to a$, equivalently
\begin{equation*}
\lim_{x\to a}\left| \frac{f(x)}{g(x)}\right| = 1.
\end{equation*}
\end{itemize}
\end{rem}

\subsection{Type-I blow-up}
In \cite{Mat} a criterion of ordinary blow-up has been stated as follows, which describes the correspondence between \lq\lq type-I blow-up" and hyperbolicity of invariant sets on $\mathcal{E}$ for (\ref{ODE-desing-dir}).

\begin{prop}[Stationary (type-I) blow-up, \cite{Mat}]
\label{prop-stationary-blowup}
Assume that (\ref{ODE-original}) has an equilibrium at infinity in the direction $x_\ast$.
Further suppose that $x_\ast$ is hyperbolic with $n_s > 1$ (resp. $n_u = n-n_s$) eigenvalues of $Jg(x_\ast)$ with negative (resp. positive) real parts.
If the solution $y(t)$ of (\ref{ODE-original}) with initial data $y(0) = y_0\in \mathbb{R}^n$ whose image $(r,x) = T(y)$ is on $W^s(x_\ast; g)$ for $g$, then $t_{\max} < \infty$ holds; namely, $y(t)$ is a blow-up solution\footnote{
The assumption of our statements implies that the global trajectory $(r,x) = T(y)$ for $g$ is {\em not} on the horizon $\mathcal{E}$.
There is a possibility that the initial data $x_0$ for $g$ is on the horizon $\mathcal{E}$, but this situation is out of focus in the proposition.
The similar caution can be said to Proposition \ref{prop-periodic-blowup}.
}.
Moreover, for a generic constant $C$, 
\begin{equation*}
r(t)^{-1} \sim C(t_{\max} - t)^{-1/k}\quad \text{ as }\quad t\to t_{\max}.
\end{equation*}
Finally, if the $i$-th component $(x_\ast)_i$ of $x_\ast$ with $i\in I_\alpha$ is not zero, then we also have
\begin{equation*}
y_i(t) \sim ({\rm sgn}x_\ast )C(t_{\max} - t)^{-\alpha_i /k}\quad \text{ as }\quad t\to t_{\max}.
\end{equation*}
In particular, $y(t)$ is an ordinary blow-up solutions.
\end{prop}


The characterization of blow-up solutions from the viewpoint of dynamical systems yields further characterization of blow-up solutions corresponding to trajectories for desingularized fields asymptotic to {\em nontrivial invariant sets} (on the horizon) such as periodic orbits, which is shown in \cite{Mat} as follows.

\begin{prop}[Periodic (type-I) blow-up, \cite{Mat}]
\label{prop-periodic-blowup}
Suppose that $g$ admits a periodic orbit $\gamma_\ast = \{x_{\gamma_\ast}(\tau)\} \subset \mathcal{E}$, with period $T_\ast > 0$, characterized by a fixed point of the Poincar\'{e} map $P:\Delta \cap \overline{\mathcal{D}}\to \Delta \cap \overline{\mathcal{D}}$.
Let $x_\ast\in \Delta\cap \gamma_\ast$; namely, $P(x_\ast)=x_\ast$.
We further assume that all eigenvalues of Jacobian matrix $JP(x_\ast)$ have moduli away from $1$
(namely, $\gamma_\ast$ is hyperbolic), at least one of which has the modulus less than $1$.
\par
If the solution $y(t)$ of (\ref{ODE-original}) whose image $(r,x) = T(y)$ is on $W^s(\gamma_\ast; g)$ for $g$, then $t_{\max} < \infty$ holds; namely, $y(t)$ is a blow-up solution.
Moreover, for a generic constant $C$, 
\begin{equation*}
r(t)^{-1} \sim C(t_{\max} - t)^{-1/k}\quad \text{ as }\quad t\to t_{\max}.
\end{equation*}
Finally, if the $i$-th component $(x_\ast)_i$ of $x_\ast$ with $i\in I_\alpha$ is not zero, then we also have
\begin{equation*}
y_i(t) \sim C(t_{\max} - t)^{-\alpha_i /k}x_i(-c'\ln (t_{\max} - t))\quad \text{ as }\quad t\to t_{\max}
\end{equation*}
for some constants $C\in \mathbb{R}$ and $c' > 0$.
\end{prop}

\begin{rem}\rm
We shall say that the blow-up is of {\em type-I} if the order of (\ref{ODE-original}) uniquely determines the asymptotic behavior of solution near blow-up. 
More precisely, $y_i(t) \sim C(t_{\max}-t)^{-\alpha_i/k}$ holds for some $C\in \mathbb{R}\setminus \{0\}$ as $t\to t_{\max}$, for all $i\in I_\alpha$ as long as $(x_\ast)_i \not = 0$.
This is compatible with cases for PDEs.
For example, it is well-known that, if we consider the initial value problem
\begin{equation}
\begin{cases}
u_t = u_{xx} + u^p, & (t,x)\in (0,T)\times \mathbb{R},\\
u(0,x) = u_0(x), & x\in \mathbb{R}
\end{cases}
\end{equation}
for some $p>1$, then for large initial data the solution $u(t,x)$ blows up with $\|u(t,\cdot)\|_{L^\infty} \sim C(t_{\max} - t)^{-1/(p-1)}$ as $t \to t_{\max}$.
Our statement corresponds to $k = p-1$.
\end{rem}






\section{A universal mechanism of blow-up and extinction in terms of center-(un)stable manifolds}
\label{section-mechanism}

The preceding results for blow-up solutions show that {\em hyperbolicity} of invariant sets induces {\em type-I} blow-up solutions, which are characterized by trajectories on the stable manifolds of hyperbolic equilibria or periodic orbits for desingularized vector fields.
On the other hand, if a given invariant set $S\subset \mathcal{E}$ such as an equilibrium or a periodic orbit is not hyperbolic, there are remaining problems whether trajectories which are on $W^{c}(S)$, namely on {\em center} manifolds, correspond to blow-up solutions for the original system.
Even if solutions turn out to be blow-up solutions, there are still nontrivial problems what rates of blow-ups are.
In such cases, linearization theory is not sufficient to analyze the asymptotic behavior since trajectories for desingularized vector fields are assumed to be on center manifolds.
On the other hand, the asymptotic behavior of trajectories on center manifolds will play key roles in geometric analysis of blow-up solutions beyond type-I blow-ups.
\par
In what follows, we consider stationary divergent solutions on {\em two-dimensional vector fields}\footnote{
We can generalize the following arguments to the higher dimensional cases, but we omit the detail because our applications in the present paper essentially treat two-dimensional problems.
}.
Assume that $x_\ast\in \mathcal{E}$ be an equilibrium for the desingularized vector field associated with (\ref{ODE-original}).
Our focus is the following type of equilibria.
\begin{dfn}[e.g., \cite{DLA2006}]\rm
Let $\dot x = g(x)$ be a smooth vector field and $\bar x$ be an equilibrium.
The point $\bar x$ is said to be (i) {\em hyperbolic} if the two eigenvalues of $Dg(\bar x)$ are away from the imaginary axis, (ii) {\em semi- (or partially) hyperbolic} if exactly one eigenvalue of $Dg(\bar x)$ is equal to zero. 
Hyperbolic or semi-hyperbolic equilibria are called of {\em elementary type}.
On the other hand, $\bar x$ is said to be (iii) {\em nilpotent} if both eigenvalues of $Dg(\bar x)$ are zero, but $Dg(\bar x)\not \equiv 0$.
Finally, it is said to be (iv) {\em linearly zero} if $Dg(\bar x)\equiv 0$.
\end{dfn}
There are several other type of equilibria called {\em center} or {\em focus} type, but we do not refer to these types here.

\subsection{Linearly zero type}
If the origin of a vector field $x'=g(x)$ is a linearly zero equilibrium, then the vector field has the following form:
\begin{equation*}
\begin{cases}
x_1' = x_1^{m_1}x_2^{n_1} \tilde g_1(x), & \\
x_2' = x_1^{m_2}x_2^{n_2} \tilde g_2(x), & 
\end{cases},\quad {}' = \frac{d}{dt},
\end{equation*}
where $\tilde g_1$ and $\tilde g_2$ are smooth functions.
Then the time-variable transform
\begin{equation*}
\frac{d\tau}{dt} = x_1^{\tilde m} x_2^{\tilde n},\quad \tilde m = \min\{m_1, m_2\},\ \tilde n = \min\{n_1, n_2\}
\end{equation*}
gives the new time-scale vector field
\begin{equation*}
\begin{cases}
\dot x_1 = x_1^{m_1-\tilde m}x_2^{n_1-\tilde n} \tilde g_1(x), & \\
\dot x_2 = x_1^{m_2-\tilde m}x_2^{n_2-\tilde n} \tilde g_2(x), &
\end{cases} ,\quad \dot {} = \frac{d}{d\tau}
\end{equation*}
possesses the origin as either being of hyperbolic, {\em semi-hyperbolic} or {\em nilpotent} type singularity mentioned below.

\subsection{Semi-hyperbolic type}
\label{section-semi-hyp}
Assume that the desingularized vector field of (\ref{ODE-original}) near an equilibrium on the horizon $x_\ast\in \mathcal{E}$ has the following form (cf. \cite{DLA2006}):
\begin{equation}
\label{prob-semi}
\begin{cases}
\dot x_1 = -x_1G(x_1,x_2) &\\
\dot x_2 = \lambda x_2 + F(x_1,x_2) &
\end{cases},\quad \dot{} = \frac{d}{d\tau},
\end{equation}
with $\lambda \in \mathbb{R}\setminus \{0\}$, where $x_\ast = (0,0)$ in the corresponding coordinate $(x_1,x_2)$, $G(0,0) = 0$, and $k+1$ is the order of the asymptotically quasi-homogeneous vector field $f$.
In this case, the {\em center manifold theorem} (e.g., \cite{C1981, CAA1980}) shows that (\ref{prob-semi}) possess a (local) center manifold $W^c(0,0) = \{x_2 = h(x_1)\}$, $|x_1| < \delta$, on which the dynamics is governed by the one-dimensional vector field
\begin{equation*}
\dot x_1 = -x_1 G(x_1, h(x_1)),
\end{equation*}
such that $h(0) = (dh/dx_1)(0) = 0$.
In particular, $h(x_1) = O(x_1^2)$.
Without the loss of generality, we may assume that $G(x_1, h(x_1)) = x_1^{d-1} + o(x_1^{d-1})$ for some $d\geq 2$.
In this case, the dynamics on $W^c(0,0)$ is governed by
\begin{equation}
\label{semi-asym}
\dot x_1 = -x_1^d + o(x_1^d)\quad \text{ as }x_1\to 0.
\end{equation}
For sufficiently small but positive initial data $x_1(0)$, the solution of (\ref{semi-asym}) can be assumed to be positive for $\tau\geq 0$ and converges to $0$ as $\tau\to \infty$.
Then the L'H$\hat{o}$pital's rule shows that
\begin{equation}
\label{ex_LHopital}
-1 = \lim_{\tau\to \infty} \frac{\dot x_1}{x_1^d} = \lim_{\tau\to \infty} t^{-1} \int_1^{x_1(\tau)} s^{-d} ds.
\end{equation}
If $w(\tau)$ be the solution of 
\begin{equation}
\label{semi-asym-w}
\dot w = -w^d,\quad w(0) = 1,
\end{equation}
then we have $\tau = -\int_1^{w(\tau)} s^{-d}ds \equiv w^{-1}(w(\tau);w(0)=1)$ and from (\ref{ex_LHopital}) we also have
\begin{equation*}
w^{-1}(x_1(\tau)) = \tau + o(\tau)\quad \text{ as }t\to \infty.
\end{equation*}
Thus we have $x_1(t) = w(\tau+o(\tau))$ as $\tau\to \infty$, which means that the asymptotic behavior of $x_1$ is described by that of $w$.
\par
Since 
\begin{equation*}
w(\tau) = \frac{1}{\sqrt{d-1}}\tau^{-\frac{1}{d-1}} + C\tau^{-\frac{d}{d-1}} + o(\tau^{-\frac{d}{d-1}})
\end{equation*}
with a constant $C$, then we finally have
\begin{equation}
\label{semi-asym-final}
x_1(\tau) =  \frac{1}{\sqrt{d-1}}\tau^{-\frac{1}{d-1}} + o(\tau^{-\frac{1}{d-1}}).
\end{equation}
This argument shows that the asymptotic behavior of trajectories converging to $S_\infty$ can be described as those for the vector field given by the principal part of the original system.
Even for high dimensional systems, the center manifold reduction (e.g., \cite{C1981}) enables us to apply the same arguments.
\par
Another straightforward way is to calculate the transform $X : x_1\mapsto w$ so that (\ref{semi-asym}) is transformed into (\ref{semi-asym-w}).
Once we obtain such a transform, the equation (\ref{semi-asym-w}) can be solved explicitly, and finally we obtain the original asymptotic behavior like (\ref{semi-asym-final}).

\subsection{Nilpotent type}
If an equilibrium for $x' = g(x)$ or the corresponding transformed vector field is nilpotent, one of useful ways to analyze the local dynamics around it to apply {\em blowing-up} technique (e.g., \cite{D1993}), which we shall call it {\em desingularization} of singularities.
In many cases, a systematic way to unravel the precise degenerate structure of equilibria so that an appropriate desingularization can be chosen is to use {\em Newton diagrams} reviewed below.
Consider the (formal) expansion of vector field
\begin{equation}
X(\bar{x}) = \sum_{j=1}^n \sum_{\bar m + \bar 1 \in \mathbb{N}^n} a_{\bar m j} \bar x^{\bar m} x_j \frac{\partial }{\partial x_j}
\end{equation}
with $X\in \chi(\mathbb{R}^n)$ and $X(0) = 0$.
Let the {\em support} of $X$ be the set of multi-indices
\begin{equation*}
S := \{\{\bar m + \bar 1\} \in \mathbb{N}^n \mid \exists j \text{ such that }a_{\bar m j}\not = 0\}.
\end{equation*}

\begin{dfn}[Newton diagrams, e.g., \cite{D1993}]\rm
Let $X$ be a vector field with the above (formal) expansion.
\begin{enumerate}
\item The {\em Newton polyhedron} of $X$ is the set $\Gamma$ defined as the convex envelop of the set $P$ given by
\begin{equation*}
P = \bigcup_{\bar k \in S} \{\bar k + \mathbb{R}_+^n\},\quad  \mathbb{R}_+^n = [0,\infty)^n.
\end{equation*}
\item The {\em Newton diagrams} of $X$ is the union $\gamma$ of compact faces $\gamma_j$ of the Newton polyhedron $\Gamma$.
\item The {\em principal part} of $X$ is the vector field $X_\Delta$ which has the following expansion:
\begin{equation}
X_\Delta(\bar{x}) = \sum_{j=1}^n \sum_{\bar m + \bar 1\in \gamma} a_{\bar m j} \bar x^{\bar m} x_j \frac{\partial }{\partial x_j}.
\end{equation}
\item The {\em quasi-homogeneous part of $X_\Delta$ associated with $\gamma_j$} is the vector field $X_j$ given by
\begin{equation}
X_j(\bar{x}) = \sum_{j=1}^n \sum_{\bar m + \bar 1\in \gamma_j} a_{\bar m j} \bar x^{\bar m} x_j \frac{\partial }{\partial x_j}.
\end{equation}
\end{enumerate}
\end{dfn}

After application of compactifications (of type $\alpha$), equilibria on the horizon are just bounded equilibria for desingularized vector fields.
Therefore, up to translation of points, the Newton diagram can be applied to desingularized vector fields near equilibria for detecting the quasi-homogeneous component\footnote{
Here the initial type $\alpha$ of asymptotically quasi-homogeneous vector field $f$ is assumed to be known in advance.
The detection of appropriate $\alpha$ remains open for general systems.
}.
Once the Newton diagram is obtained, the desingularization of degenerate equilibria follows from the standard treatment in singularity theory.
Hopefully, equilibria for the desingularized system become hyperbolic, and the preceding approach can be applied.

\subsection{Strategy for detecting blow-up rates}

Here we gather the approach how we detect rigorous blow-up rates of blow-up solutions.
Our approach is based on the study of invariant sets on the horizon for desingularized vector fields and appropriate transformations of vector fields with smooth and orbital equivalence of dynamical systems.

\begin{dfn}[Equivalence]\rm
Consider flows $\phi$ and $\psi$ generated by differential equations $x'=f(x)$ and $y'=g(y)$ on phase spaces $X$ and $Y$ ($\mathbb{R}^n$ or its open subsets), respectively.
We say that $\phi$ and $\psi$ are {\em smoothly equivalent} if there is a diffeomorphism $h:X\to Y$ such that 
\begin{equation*}
f(x) = M^{-1}(x)g(h(x))\quad \text{ where }\quad M(x) = \frac{dh(x)}{dx}.
\end{equation*}
In that case, flows $\phi$ and $\psi$ are transformed into each other by the coordinate transform $y=h(x)$. 
In particular, $\phi$ and $\psi$ determines the same dynamics.
\par
On the other hand, we say that two flows $\phi_1$ and $\phi_2$ generated by smooth differential equations $\dot x = f_i(x)\ (i=1,2)$ are said to be {\em orbitally equivalent} if there is a positive (smooth) function $\mu : X\to \mathbb{R}$ such that $f_1(x) = \mu(x)f_2(x)$.
This relationship means that trajectories for $\phi_1$ and $\phi_2$ differ only in time parameterization.
\end{dfn}

A strategy for detecting blow-up rates is based on explicit representation of trajectories on center-stable manifolds of invariant sets on the horizon, possibly with the assistance of smooth equivalence and orbital equivalence.

\begin{description}
\item[Step 1] : Apply the compactification associated with the type of asymptotically quasi-homogeneous vector fields to the original problem, and detect invariant sets $S_\infty$ on the horizon for desingularized vector fields.
\end{description}
This step is fundamental and the same as in the preceding works for studying dynamics at infinity (e.g., \cite{DH1999}).
Here we assume that $S_\infty$ is either an equilibrium or a periodic orbit.
\begin{description}
\item[Step 2] : Study the linear stability of $S_\infty$. 
If $S_\infty$ is hyperbolic, the blow-up rate is given by results stated in Section \ref{section-type-I}.
\end{description}
This step is involved in Section \ref{section-type-I}.

\begin{description}
\item[Step 3] : Assume that a trajectory for the desingularized vector field is on $W^{cs}(S_\infty)$.
Then we transform the vector field with further time-variable desingularizations and orbital (or smooth) equivalences so that the transformed (desingularized) vector field for variables, say $(r,x)\in \mathbb{R}^{1+(n-1)}$, can be {\em explicitly} solved or is topologically conjugate to the linearization near $S_\infty$ (via e.g., Hartman-Grobman's theorem, e.g., \cite{Rob}).
\end{description}

We know that positive functions giving orbital equivalence can trigger blow-up rates of blow-up solutions different from type-I rate.
Our arguments refer to asymptotic behavior in {\em critical cases} and, in such cases, the arguments shown in Section \ref{section-semi-hyp} describes the asymptotic behavior of trajectories in terms of those for {\em simpler} dynamical systems.


\begin{description}
\item[Step 4] : Calculate the asymptotic behavior $t_{\max}-t$ from the total time-variable desingularization $dt/d\eta = T(r(\eta),x(\eta))$ and the asymptotic form of $r(\eta)$ and $x(\eta)$ near $S_\infty$.
\end{description}

The concrete strategy depends on problems and structure of invariant sets on the horizon for desingularized vector fields.
Nevertheless, we will see that the above methodology will give a fundamental guideline for detecting dynamics of blow-up solutions including blow-up rates.

\subsection{Extinction and compacton traveling waves}
Next we consider the ordinary differential equation in $\mathbb{R}^n$: 
\begin{equation*}
x' = f(x).
\end{equation*}
Our present interest is a solution $x(t)$ such that $x(t)$ arrives at a point $p\in \mathbb{R}^n$ at a time $t=T$ and that {\em $x(t)\equiv p$ holds for all $t\geq T$}.
Such a kind of solutions can be observed in terms of {\em finite-time extinction} for degenerate or nonlinear-diffusive parabolic systems.
For example, consider the traveling wave equation of 
\begin{equation*}
u_t = d(u^{m+1})_{xx} + f(u),\quad (t,x)\in [0,\infty)\times \mathbb{R}
\end{equation*}
with $m>0$ and smooth $f : \mathbb{R}\to \mathbb{R}$.
The assumption $m > 0$ is crucial in this case, otherwise the above equation is just a reaction-diffusion equation if $m=0$.
In particular, the equation degenerates at $u=0$, and hence the system admits a singularity at $u=0$\footnote{
This point is different from blow-up problems, since blow-up problems are typically considered in {\em smooth} differential equations.
}.
Details are discussed in Section \ref{section-ex-extinction}.

\begin{rem}
\label{rem-slogan}
We gather the similarity between blow-up problems and extinction problems.
\begin{itemize}
\item The vector fields have singularity, which corresponds to invariant sets $S$, such as equilibria, for associated {\em desingularized} vector fields.
\item The targeting singular solutions correspond to those asymptotic to $S$; in particular, trajectories on {\em center-stable manifolds $W^{cs}(S)$ of $S$}, for desingularized vector fields.
\item The maximal existence time of singular solutions can be calculated by time-variable desingularizations determining desingularized vector fields.
\item The asymptotic rates near singularities can be characterized by {\em dynamics on $W^{cs}(S)$} and time-variable desingularizations.
\end{itemize}
\end{rem}

We see in the following examples as well as preceding studies that finite-time singularities can be characterized according to slogans in Remark \ref{rem-slogan}.

\section{Concrete asymptotic behavior}
\label{section-example}
We demonstrate several examples and compute rigorous blow-up rates of blow-up solutions and extinction rates of extinction solutions.
The main aim in this section is to show that the dynamics on center-(un)stable manifolds actually determines asymptotic behavior near blow-up as well as extinction.
Combining with rigorous asymptotic rates derived in previous works, our approach reveals a comprehensive understanding of finite-time asymptotics depending on internal data such as order of nonlinearity and coefficients of vector fields.
Let $C>0$ be a generic constant which is used in various estimates below.
%
%
\subsection{Anada-Ishiwata-Ushijima's example : Blow-up}
The first example is 
\begin{equation}
\label{AIU}
\frac{du_0}{dt} = u_1 u_0^{-2},\quad \frac{du_1}{dt} = u_1^2 u_0^{-1}.
\end{equation}
There is a following fact about blow-up solutions of (\ref{AIU}) shown in \cite{AIU}.
\begin{fact}[\cite{AIU}]
\label{fact-AIU}
For positive initial data $u_0(0) > 0$ and $u_1(0) > 0$, the solution $(u_0(t), u_1(t))$ blows up at $t = t_{\max}$ with the following rate:
\begin{equation*}
u_0(t) = O((\log(t_{\max}-t)^{-1})^{1/2}),\quad u_1(t) = O((t_{\max}-t)^{-1}(\log(t_{\max}-t)^{-1})^{1/2})\text{ as }t\to t_{\max}.
\end{equation*}
\end{fact}
We try to detect blow-up rates from the viewpoint of dynamics around equilibria at infinity.
The direct observation yields the following property of vector fields.
\begin{lem}
The vector field (\ref{AIU}) is quasi-homogeneous of type $(0,1)$ and order $2$.
\end{lem}
Following the type of vector fields, introduce the following directional compactification:
\begin{equation}
\label{cpt-1}
u_0 = x,\quad u_1 = \frac{1}{r}.
\end{equation}
Then the translated vector field is
\begin{equation*}
\frac{dx}{dt} = r^{-1} x^{-2},\quad \frac{dr}{dt} = -x^{-1}.
\end{equation*}
Introducing the time variable desingularization
\begin{equation*}
\frac{dt}{d\tau} = rx^2,
\end{equation*}
we have the following desingularized vector field:
\begin{equation*}
\frac{dx}{d\tau} = 1,\quad \frac{dr}{dt} = -rx.
\end{equation*}
Note that the function $rx^2$ is positive for $r> 0$ and $x\not = 0$.
We then solve the system to obtain
\begin{equation}
\label{sol-1}
x(\tau) = \tau+C,\quad r(\tau) = C' e^{-\frac{(\tau+C)^2}{2}}
\end{equation}
for some $C\in \mathbb{R}$ and $C' > 0$.
The maximal existence time $t_{\max}$ of solutions in $t$-scale is computed by
\begin{equation*}
t_{\max} = \int_0^\infty C' e^{-\frac{(\eta+C)^2}{2}}(\eta+C)^2d\eta ,
\end{equation*}
which is finite and therefore the corresponding solution is a blow-up solution.
The asymptotic behavior of $t$ near $t_{\max}$ is given as the following estimate:
\begin{equation*}
t_{\max}-t = \int_\tau^\infty C' e^{-\frac{(\eta+C)^2}{2}}(\eta+C)^2 d\eta \sim C' e^{-\frac{(\tau+C)^2}{2}}\tau\quad \text{ as }\quad \tau\to \infty\ (\Leftrightarrow t\to t_{\max}).
\end{equation*}
In particular, we have
\begin{equation*}
\tau \sim C(\log(t_{\max}-t)^{-1})^{1/2}.
\end{equation*}
Inserting this asymptotics to (\ref{cpt-1}) as well as (\ref{sol-1}), we have
\begin{align*}
u_0 &= x \sim \tau \sim C(\log(t_{\max}-t)^{-1})^{1/2}, \\
u_1 &= \frac{1}{r} \sim Ce^{\frac{\tau^2}{2}} = C(e^{\frac{\tau^2}{2}} \tau^{-1}) \tau \sim C(t_{\max}-t)^{-1}(\log(t_{\max}-t)^{-1})^{1/2}.
\end{align*}
The above asymptotics are exactly the same as those in Fact \ref{fact-AIU}.
More precisely, we have the following result.
\begin{thm}
Consider (\ref{AIU}) the initial data $(u_0(0), u_0(1))$. 
For sufficiently large $u_1(0)$, the $u_1$-component of the solution blows up at $t=t_{\max} < \infty$ with the blow-up rate $u_1(t) \sim C(t_{\max}-t)^{-1} (\log (t_{\max}-t)^{-1})^{1/2}$, while $u_0(t)$ also blows up with the rate $u_0(t) \sim C(\log (t_{\max}-t)^{-1})^{1/2}$.
\end{thm}

%
%
\subsection{Ishiwata-Yazaki's example : Blow-up}

The next example is
\begin{equation}
\label{IY}
\frac{dv_0}{dt} = av_0^{\frac{a+1}{a}}v_1,\quad \frac{dv_1}{dt} = av_1^{\frac{a+1}{a}}v_0,
\end{equation}
where $a\in \mathbb{R}$ is a parameter. 
There is a following fact about blow-up solutions of (\ref{IY}) obtained in \cite{IY2003}.
\begin{fact}[\cite{IY2003}]
\label{fact-IY}
Set positive initial data $v_0(0) > 0$ and $v_1(0) > 0$.
If $a \in (0,1)$ and $v_0(0)\not = v_1(0)$, the solution $(v_0(t), v_1(t))$ blows up at $t = t_{\max}$ with the rate $O((t_{\max}-t)^{-a})$.
\end{fact}
This subsection aims at revealing the above rate of blow-up solutions as well as the mechanism of dynamics at infinity.
Firstly, the direct observation yields the following property.
\begin{lem}
The vector field (\ref{IY}) is homogeneous of order $2+a^{-1}$ with a natural extension of order $k\in \mathbb{N}$ to $\mathbb{R}_{\geq 1}$\footnote
{
Since we only treat the scaling for positive $r > 0$, then such an extension of \lq\lq order" is sufficient to our arguments.
In the field of algebraic geometry, this treatment is also considered in terms of {\em positive (quasi-)homogeneity}.
}.
\end{lem}
Following the homogeneity, we introduce the following directional compactification:
\begin{equation*}
v_0 = \frac{1}{r},\quad v_1 = \frac{u}{r}.
\end{equation*}
Then the transformed vector field is
\begin{equation*}
\frac{dr}{dt} = -ar^{-\frac{1}{a}}u,\quad \frac{du}{dt} = ar^{-\frac{a+1}{a}}(-u^2 + u^{\frac{a+1}{a}}).
\end{equation*}
Using the time-scale desingularization
\begin{equation*}
\frac{dt}{d\tau} = r(t)^{\frac{a+1}{a}},
\end{equation*}
we have the following desingularized vector field:
\begin{equation}
\label{IY-desing}
\frac{dr}{d\tau} = -ar u,\quad \frac{du}{d\tau} = a(-u^2 + u^{\frac{a+1}{a}}).
\end{equation}
As for dynamics at infinity, we have the following observations.
\begin{prop}
Consider (\ref{IY-desing}) with $a < 1$. Then
\begin{enumerate}
\item Points $(r,u) = (0,0), (0,1)$ are equilibria at infinity.
\item The set $\{u = 1\}$ is an invariant manifold of (\ref{IY-desing}).
\item The equilibrium at infinity $(0,1)$ is a saddle. In particular, it is stable in $r$-direction and unstable in $u$-direction.
\end{enumerate}
\end{prop}
Stationary Blow-up Theorem (Proposition \ref{prop-stationary-blowup}) indicates if the initial data $(v_0(0), v_1(0))$ are set so that $v_0(0)= v_1(0) > 0$ then the solution blows up with the blow-up rate $O((t_{\max}-t)^{\frac{-a}{a+1}})$.
The blow-up rate is different from Fact \ref{fact-IY}, while it is consistent in assumptions on initial data.
Moreover, the equilibrium $(0,1)$ being a saddle indicates that the blow-up solution with the blow-up rate $O((t_{\max}-t)^{\frac{-a}{a+1}})$ is unstable with respect to initial data.
\par
Now pay attention to the equilibrium $(0,0)$ for (\ref{IY-desing}), which is linearly zero.
Now introduce another time-scale desingularization
\begin{equation*}
\frac{d\eta}{d\tau} = u.
\end{equation*}
Note that $\{u=0\}$ is an invariant manifold for (\ref{IY-desing}), which implies that $u$ is always positive if $u(0) > 0$.
\par
\bigskip
We thus have
\begin{equation}
\label{IY-desing2}
\frac{dr}{d\eta} = -ar,\quad \frac{du}{d\eta} = a(-u + u^{\frac{1}{a}}),
\end{equation}
in which case the origin $(0,0)$ is hyperbolic.
Thus the solution is written as
\begin{equation*}
r(\eta) = Ce^{-a\eta},\quad u(\eta) = Ce^{-a\eta}(1+o(1)).
\end{equation*}
Note that, if $u(0) < 1$, then the solution goes to the origin.
The maximal existence time $t_{\max}$ is calculated as
\begin{equation*}
t_{\max} = \int_0^\infty r^{\frac{a+1}{a}}d\tau = \int_0^\infty r^{\frac{a+1}{a}} u^{-1}d\eta = C\int_0^\infty e^{-\eta}(1+o(1))d\eta < \infty,
\end{equation*}
which implies that the solution $(v_0(t), v_1(t))$ such that $(r(\eta), u(\eta))\in W^{cs}(0,0)$ in the $\eta$-scale is a blow-up solution.
The above estimate also shows
\begin{align*}
t_{\max} - t &\sim C\int_\eta^\infty e^{-\tilde \eta} d\tilde \eta = Ce^{-\eta}\quad \text{ as }\quad \eta\to \infty.
\end{align*}
Therefore we finally have
\begin{equation*}
v_0 = \frac{1}{r}= Ce^{a\eta} = C(e^{-\eta})^{-a} \sim C(t_{\max}-t)^{-a}.
\end{equation*}
In particular, the solution $(v_0(t), v_1(t))$ blows up at $t=t_{\max}$ with the rate $O((t_{\max}-t)^{-a})$ if $v_0(0) > v_1(0)$.
Symmetry of (\ref{IY}) implies that the same result holds if $v_0(0) < v_1(0)$ applying the directional compactification
\begin{equation*}
v_0 = \frac{u}{r},\quad v_1 = \frac{1}{r}.
\end{equation*}
Summarizing the above arguments, we have the following result.
\begin{thm}
Consider (\ref{IY}) with the initial data $v_0(0) > 0$ and $v_1(0) > 0$.
Assume that $a < 1$.
If $v_0(0) = v_1(0)$, then the solution blows up with the blow-up rate $O((t_{\max}-t)^{\frac{-a}{a+1}})$.
If $v_0(0) \not = v_1(0)$, then the solution blows up with the blow-up rate $O((t_{\max}-t)^{-a})$.
\end{thm}

%
%
\subsection{Andrews' example 1 : Blow-up}
The next example is
\begin{equation}
\label{Andrews1}
\begin{cases}
\displaystyle{
\frac{du_0}{dt} = \frac{1}{\sin \theta} u_1u_0^2 - \frac{2a\cos\theta}{\sin \theta}u_0^3
}
, & \\
\displaystyle{
\frac{du_1}{dt} =  \frac{a}{\sin \theta} u_0u_1^2 - \frac{1-a}{\sin \theta} \frac{u_0u_1^3}{u_1 + 2\cos\theta u_0}
}, &
\end{cases}
\end{equation}
where $a\in (0,1)$ and $\theta\in (0,\pi/2)$ are parameters.
This system is originally studied in \cite{A2002}. See also \cite{AIU}.

\begin{fact}[cf. \cite{AIU, A2002}]
Consider (\ref{Andrews1}) with the initial data $(u_0(0), u_1(0))$ which are both positive.
\begin{enumerate}
\item If $a < 1/2$, the solution with sufficiently large positive initial data blows up at $t=t_{\max} < \infty$ with the rate $(t_{\max}-t)^{-1}$ as $t\to t_{\max}$.
\item If $a\in (1/2, 1)$, there is a blow-up with the blow-up rate $\{(t_{\max}-t)^{-1} \log (t_{\max}-t)^{-1}\}^{1/2}$ as $t\to t_{\max}$.
\item If $a= 1/2$, there is a blow-up solution with the blow-up rate $\{(t_{\max}-t)^{-1} (\log (t_{\max}-t)^{-1})^{1/2}\}^{1/2}$
as $t\to t_{\max}$.
\end{enumerate}
\end{fact}
This preceding result contains the following non-trivial features: (i) {\em existence of blow-up solutions}, (ii) {\em difference of blow-up rates among three regimes: $a\in (0,1/2)$, $a \in (1/2, 1)$ and $a = 1/2$}, and (iii) {\em discontinuous change of such blow-up rates}.
Now we consider blow-up solutions of (\ref{Andrews1}) following our strategy, which reveals the above non-trivial nature of blow-up behavior.
\par
First, we directly know the following property.
\begin{lem}
The vector field (\ref{Andrews1}) is homogeneous with order $k+1 = 3$. 
\end{lem}
Introducing the directional compactification
\begin{equation*}
u_0 = \frac{x}{r},\quad u_1 = \frac{1}{r},
\end{equation*}
we have 
\begin{align*}
r' &= \frac{-r^{-1}x}{\sin \theta}\left(a + \frac{1-a}{1+2\cos\theta x}\right),\\
x' &= \frac{x^2}{r^2\sin\theta}\left( -a - \frac{1-a}{1+2\cos\theta x} + (1-2a\cos\theta x)\right).
\end{align*}
Under the time scale desingularization
\begin{equation*}
\frac{dt}{d\tau} = r^2(1+2\cos\theta x),
\end{equation*}
we also have
\begin{equation}
\label{And-1-desing}
\begin{cases}
\displaystyle{
\dot r = \frac{-r x}{\sin \theta}\left( 1+2a\cos\theta x \right)
},& \\
\displaystyle{
\dot x = \frac{2x^2}{\sin\theta}\left\{ \cos \theta(1-2a) x -2a\cos^2\theta x^2\right\}
}, & \quad 
\displaystyle{
\dot {} = \frac{d}{d\tau}
}.
\end{cases}
\end{equation}
Note that the exponent $2$ in the above time-scale desingularization comes from the order $k+1=3$ of (\ref{Andrews1}).
The standard analysis of dynamics at infinity indicates the following observation.
\begin{prop}
Consider (\ref{And-1-desing}). Then the following statements hold.
\begin{enumerate}
\item For each $\theta\in (0,\pi/2)$, (\ref{And-1-desing}) admits the following equilibria at infinity:
\begin{equation*}
(0,0) \text{ for all $a\in (0,1)$ },\quad \left(0, \frac{1-2a}{2a\cos\theta} \right) \text{ if }a\not = \frac{1}{2}.
\end{equation*}
\item The equilibrium $(0,0)$ is linearly zero for all $a\in (0,1)$. Another equilibrium $\left(0,\frac{1-2a}{2a\cos\theta} \right)$ is a sink if $a\in (0, 1/2)$, a saddle if $a\in (1/2,1)$, in which case it is stable in $r$-direction and unstable in $x$-direction.
\end{enumerate}
\end{prop}
This proposition indicates that, for generic positive (large) initial data $(v_0(0), v_1(0))$, the solution blows up with the rate $O((t_{\max}-t)^{-1/2})$ as $t\to t_{\max}$ when $a < 1/2$.
\begin{rem}
Note that the function $r^2 (1+2\cos\theta x)$ is positive near $(0,0)$ and near $(0, (1-2a) / 2a\cos \theta)$ if $a < 1/2$.
On the other hand, if $a > 1/2$, the function $r^2 (1+2\cos\theta x)$ is negative near $(0, (1-2a) / 2a\cos \theta)$, which indicates that the original problem (\ref{Andrews1}) and the desingularized vector field (\ref{And-1-desing}) are not orbitally equivalent near the saddle $(0, (1-2a) / 2a\cos \theta)$ if $a > 1/2$.
\end{rem}
Now we move to the case $a\geq 1/2$, which is of our main interest.
The origin is linearly zero. 
Note that all terms in the vector field (\ref{And-1-desing}) can be divided by the factor $x$.
Therefore we introduce further time-scale transformation
\begin{equation}
\label{time-scale-x2eta}
\frac{d\eta}{d\tau} = x,
\end{equation}
which leads to another vector field
\begin{align}
\label{And-1-desing2-1}
\frac{dr}{d\eta} &= \frac{-r}{\sin \theta}\left( 1+2a\cos\theta x \right),\\
\label{And-1-desing2-2}
\frac{dx}{d\eta} &= \frac{x}{\sin\theta}\left\{ 2\cos\theta(1-2a)x - 4a\cos^2 \theta x^2\right\}.
\end{align}
\begin{rem}
The total time-variable desingularization
\begin{equation*}
\frac{dt}{d\eta} = \frac{dt}{d\tau}\frac{d\tau}{d\eta} = r^2(1+2\cos\theta x)x^{-1} 
\end{equation*}
becomes always positive near $(0,0)$ but $r\not =0, x> 0$, and near the saddle $(0, (1-2a) / 2a\cos \theta)$ for $a\not = 1/2$.
Then the vector field (\ref{And-1-desing2-1})-(\ref{And-1-desing2-2}) is orbitally equivalent to (\ref{Andrews1}) near equilibria on the horizon for all $a\in (0,1)$.
\end{rem}

\begin{description}
\item[Case 1 : $a\in (1/2, 1)$.]
\end{description}
In this case $1-2a < 0$ and hence the $x^2$-term in (\ref{And-1-desing2-2}) does not vanish.
The origin turns out to be semi-hyperbolic for (\ref{And-1-desing2-1})-(\ref{And-1-desing2-2}).
Now the $x$-component (\ref{And-1-desing2-2}) is a closed system in $x$.
By arguments in Section \ref{section-semi-hyp}, it turns out that the solution $x(\eta)$ of (\ref{And-1-desing2-2}) with sufficiently small positive initial data tends to zero and we have
\begin{align*}
x(\eta) &= \bar x(\eta + o(\eta))\quad \text{ as }\eta\to \infty,\\
\bar x(\eta) &= \left(\frac{2(2a-1)}{\tan\theta} \eta + \frac{1}{x_0}\right)^{-1},\quad \bar x(0) = x_0,
\end{align*}
where $\bar x$ solves 
\begin{equation*}
\frac{d\tilde x}{d\eta} = \frac{2(2a-1)}{\tan\theta}\tilde x^2,\quad \tilde x(0) = x_0.
\end{equation*}
Note that the right-hand side of the equation is the least order term of that in (\ref{And-1-desing2-2}).

Therefore (\ref{And-1-desing2-1}) becomes
\begin{equation*}
\frac{dr}{r} = -\frac{1}{\sin\theta}\left\{ 1+2a\cos\theta \left(\frac{2(2a-1)}{\tan\theta} \eta(1+o(1)) + \frac{1}{x_0}\right)^{-1} \right\}d\eta.
\end{equation*}
Integrating both sides, we have
\begin{equation*}
r(\eta) = \frac{Ce^{-\eta/\sin\theta}}{\eta + C'}.
\end{equation*}

\par
\bigskip
Now we have
\begin{equation*}
\frac{dt}{d\eta} = \frac{dt}{d\tau}\frac{d\tau}{d\eta} = r^2(1+2\cos\theta x)x^{-1} \sim Ce^{-2\eta/\sin\theta}\eta^{-1}(1+o(1))
\end{equation*}
and hence the maximal existence time is
\begin{equation*}
t_{\max} = C\int_0^\infty e^{-2\eta/\sin\theta}\eta^{-1} (1+o(1)) d\eta < \infty.
\end{equation*}
The asymptotic behavior is
\begin{equation*}
t_{\max}-t = C\int_{\eta}^\infty e^{-2\eta /\sin\theta}\eta^{-1} (1+o(1)) d\eta \sim C\eta^{-1} e^{-2\eta/\sin \theta}\quad \text{ as }\quad \eta \to \infty
\end{equation*}
and the least order asymptotics of $\eta$ is 
\begin{equation*}
\eta \sim C\ln ((t_{\max}-t)^{-1})\quad \text{ as }\quad \eta \to \infty.
\end{equation*}
We finally have
\begin{align*}
u_1 &= \frac{1}{r} \sim Ce^{\eta/\sin \theta}\eta^{-1} = C(e^{-2\eta/\sin \theta} \eta^{-1})^{-1/2} \eta^{1/2} \sim C\{(t_{\max}-t)^{-1} \log (t_{\max}-t)^{-1}\}^{1/2}, \\
u_0 &= \frac{x}{r} \sim Ce^{\eta/\sin \theta} = C(e^{-2\eta/\sin \theta} \eta^{-1})^{-1/2}\eta^{-1/2} \sim C\{(t_{\max}-t)^{-1} (\log (t_{\max}-t)^{-1})^{-1}\}^{1/2}
\end{align*}
as $t\to t_{\max}$.
\par
\bigskip
\begin{description}
\item[Case 2 : $a = 1/2$.]
\end{description}
In this case, observe that $1-2a = 0$, which indicates that the $x^2$-term in (\ref{And-1-desing2-2}) disappears.
In particular, the leading term of $x$-evolution is changed as follows:
\begin{equation*}
\frac{dr}{d\eta} = \frac{-r}{\sin \theta}\left( 1+\cos\theta x \right),\quad
\frac{dx}{d\eta} = -\frac{2\cos^2\theta}{\sin\theta}x^3.
\end{equation*}
Similar to the previous case, we obtain an explicit form of $x(\eta)$ as follows:
\begin{equation*}
x(\eta) = C'(\eta+C'')^{-1/2},\quad C', C'' >0.
\end{equation*}
Inserting $x=x(\eta)$ into the $r$-equation and applying variable separation method, we have
\begin{equation*}
r(\eta) = Ce^{-\eta/\sin \theta}e^{-(\eta + C')^{1/2}} \sim Ce^{-\eta/\sin \theta}.
\end{equation*}
Time-scale transformation satisfies $\frac{dt}{d\eta} = \frac{dt}{d\tau}\frac{d\tau}{d\eta} = r^2(1+2\cos\theta x)x^{-1}$, and hence
\begin{equation*}
t_{\max}-t = C\int_{\eta}^\infty e^{-2\eta /\sin\theta} \{ 1+ (\eta+C'')^{-1/2}\}(\eta+C'')^{1/2} d\eta \sim C\eta^{1/2} e^{-2\eta/\sin \theta}\quad \text{ as }\quad \eta \to \infty
\end{equation*}
As in the previous case, the maximal existence time $t_{\max}$ is finite.
%
We finally obtain
\begin{align*}
u_1 &= \frac{1}{r} \sim Ce^{\eta/\sin \theta} = C(e^{-2\eta/\sin \theta}\eta^{1/2})^{-1/2} \eta^{1/4} \sim C\{(t_{\max}-t)^{-1} (\log (t_{\max}-t)^{-1})^{1/2}\}^{1/2}\\
u_0 &= \frac{x}{r} \sim Ce^{\eta/\sin\theta} \eta^{-1/2} = C(e^{-2\eta/\sin\theta} \eta^{1/2})^{-1/2}\eta^{-1/4} \sim C\{(t_{\max}-t)^{-1} (\log (t_{\max}-t)^{-1})^{-1/2}\}^{1/2}
\end{align*}
as $t\to t_{\max}$.
Summarizing the above arguments, we have the following theorem.
\begin{thm}[cf. \cite{AIU, A2002}]
Consider (\ref{Andrews1}) with the initial data $(u_0(0), u_1(0))$.
\begin{enumerate}
\item If $a < 1/2$, the solution with sufficiently large positive initial data blows up at $t=t_{\max} < \infty$ with the rate $(t_{\max}-t)^{-1}$ as $t\to t_{\max}$.
\item If $a\in (1/2, 1)$ and if $u_1(0) > 0$ is sufficiently large with $u_0(0) = (1-2a)u_1(0) / (2a\cos\theta)$, then the solution blows up at $t=t_{\max}<\infty$ with the rate $(t_{\max}-t)^{-1}$ as $t\to t_{\max}$.
On the other hand, if $u_0(0) > (1-2a)u_1(0) / (2a\cos\theta)$, then the solution blows up with the following rate:
\begin{align*}
u_0 & \sim C\{(t_{\max}-t)^{-1} (\log (t_{\max}-t)^{-1})^{-1}\}^{1/2},\\
u_1 & \sim C\{(t_{\max}-t)^{-1} \log (t_{\max}-t)^{-1}\}^{1/2}
\end{align*}
as $t\to t_{\max}$.
\item If $a= 1/2$, $u_1(0) > 0$ is sufficiently large and $u_1(0) \gg u_0(0) > 0$, then the solution blows up with the following rate:
\begin{align*}
u_0 & \sim C\{(t_{\max}-t)^{-1} (\log (t_{\max}-t)^{-1})^{-1/2}\}^{1/2},\\
u_1 & \sim C\{(t_{\max}-t)^{-1} (\log (t_{\max}-t)^{-1})^{1/2}\}^{1/2}
\end{align*}
as $t\to t_{\max}$.
\end{enumerate}
\end{thm}

This result shows that our approach reveals the component-wise blow-up rates of solutions as well as the dependence of rates as a function of parameters from the viewpoint of dynamical systems.

%
%
\subsection{Andrews' example 2 : Blow-up, case study}

Consider
\begin{equation}
\label{Andrews-2}
\begin{cases}
u_1' = u_1^2(a_2 u_2 + a_3 u_3 - a_1 u_1) & \\
u_2' = u_2^2(a_3 u_3 + a_1 u_1 - a_2 u_2) & \\
u_3' = u_3^2(a_1 u_1 + a_2 u_2 - a_3 u_3) & \\
\end{cases},
\end{equation}
where $a_1,a_2,a_3$ are positive constants.
Obviously the vector field is homogeneous of order $2+1$.
In this system the following fact is known:
\begin{fact}[\cite{A2002}]
\label{fact-A2}
If $a_1 = a_2 + a_3$ and $u_1(0) > 0$, $u_2(0) = u_3(0) > 0$, then the solution blows up in a finite time $t_{\max}$ with the blow-up rate $((t_{\max}-t)^{-1}\log (t_{\max}-t)^{-1})^{1/2}$. 
\end{fact}
Our aim here is to obtain blow-up solutions with the above blow-up rate from the dynamical systems' viewpoint.
First we apply the (homogeneous) directional compactification
\begin{equation*}
u_1 = \frac{1}{r},\quad u_2 = \frac{1}{rw},\quad u_3 = \frac{y}{rw}
\end{equation*}
to (\ref{Andrews-2}), which is different from preceding examples\footnote
{
An ordinary compactification is $(u_1, u_2, u_3) = (1/r, x/r, y/r)$, for example.
It turns out that the horizon $r=0$ is an invariant manifold of corresponding desingularized vector fields, but the vector field on the horizon does not have equilibria and solutions with $x,y > 0$ diverge in general. 
Such a divergent behavior will lead to application of {\em successive} compactifications.
}.
Then we have
\begin{align*}
-r^{-2}r' &= r^{-2}\left(\frac{a_2}{rw} + \frac{a_3 y}{rw} - \frac{a_1}{r}\right)\\
\Leftrightarrow r' &= -\frac{1}{rw}\left(a_2 + a_3 y - a_1 w\right),\\
\left(\frac{1}{rw}\right)' &= -r^{-2}w^{-1}r' - r^{-1}w^{-2}w' = \frac{1}{r^2w^2}\left(\frac{a_3 y}{rw} + \frac{a_1}{r} - \frac{a_2}{rw}\right)\\
\Leftrightarrow w' &= \frac{1}{r^2}\left(a_2 + a_3 y - a_1 w\right) - \frac{1}{r^2w}\left( a_3 y + a_1 w - a_2\right),\\
\left(\frac{y}{rw}\right)' &= -r^{-2}w^{-1}y r' - r^{-1}w^{-2}y w' + (rw)^{-1} y' = \frac{y^2}{r^2w^2}\left( \frac{a_1}{r} + \frac{a_2}{rw} - \frac{a_3 y}{rw}\right)\\
\Leftrightarrow y' &= - \frac{y}{r^2w^2}\left( a_3 y + a_1 w - a_2\right) + \frac{y^2}{r^2w^2}\left( a_1 w + a_2  -  a_3 y \right).
\end{align*}
We desingularize the vector field by introducing
\begin{equation*}
\frac{d\tau}{dt} = (rw)^{-2},
\end{equation*}
which yields
\begin{equation}
\label{Andrews-2-desing}
\begin{cases}
\dot r = -rw f_1(w,y), & \\
\dot w = w \left(w f_1(w,y) - f_2(w,y) \right), & \\
\dot y =  y \left(y f_3(w,y) - f_2(w,y) \right), & 
\end{cases}\quad \dot {}=\frac{d}{d\tau}
\end{equation}
where
\begin{equation*}
f_1(w,y) = a_2 + a_3 y - a_1w,\quad f_2(w,y) = a_3 y + a_1w - a_2,\quad f_3(w,y) = a_1w + a_2 - a_3 y.
\end{equation*}
Our concern here is equilibria on $\{r=w=0\}$.
Here we further make the following assumption:
\begin{equation}
\label{ass-Andrews-Mat}
a_2 = a_3 \equiv a > 0.
\end{equation}
Then direct calculations yield the following observation.
\begin{lem}
Assume (\ref{ass-Andrews-Mat}).
Then the set $\{y=1\}$ is an invariant subset for (\ref{Andrews-2-desing}). 
\end{lem}
We then restrict our consideration to $\{y=1\}$ with (\ref{ass-Andrews-Mat}), which corresponds to solutions with the initial data $u_2(0) = u_3(0)$.
In this case, the desingularized vector field (\ref{Andrews-2-desing}) is reduced to the following:
\begin{equation*}
\begin{cases}
\dot r = -rw (2a - a_1w), & \\
\dot w = w \left(w (2a - a_1w) - a_1w \right) = w \left( (2a - a_1)w - a_1w^2 \right). &
\end{cases}
\end{equation*}
We further introduce the time-scale transform:
\begin{equation*}
\frac{d\eta}{d\tau} = w.
\end{equation*}
Then we have
\begin{equation}
\label{Andrews2-blowup}
\frac{dr}{d\eta} = -r (2a - a_1w),\quad
\frac{dw}{d\eta} = (2a - a_1)w - a_1w^2.
\end{equation}
The above form of vector field indicates that the origin $(0,0)$ is stable if $a_1 \geq 2a > 0$, but the asymptotic behavior is drastically changed where or not $a_1 = 2a$.
On the other hand, if $2a > a_1 > 0$, the system possess $(r,w) = (0,(2a-a_1)/a_1)$ as another meaningful equilibrium on the horizon.
\begin{description}
\item[Case 1 : $2a > a_1 > 0$.]
\end{description}
In this case, the equilibrium $(0,(2a-a_1)/a_1)$ is a sink, where eigenvalues of the linearized matrix are $\{-a_1, -2a+a_1\}$.
Consider trajectories asymptotic to $(0,(2a-a_1)/a_1)$.
In this case, we have
\begin{align*}
t_{\max} &= \int_0^\infty \frac{dt}{d\tau}d\tau = \int_0^\infty \frac{dt}{d\tau}\frac{d\tau}{d\eta} d\eta\\
	&= \int_0^\infty r(\eta)^2 w(\eta) d\eta = C\int_0^\infty e^{-2a_1\eta}(1+o(1)) d\eta < \infty.
\end{align*}
Then the corresponding solutions in the original problem (\ref{Andrews-2}) blow up in finite time.
Moreover, 
\begin{align*}
t_{\max}-t &\sim C\int_\eta^\infty e^{-2a_1\tilde \eta}(1+o(1)) d\tilde \eta \sim Ce^{-2a_1\eta}(1+o(1)) \quad \text{ as }\quad \eta \to \infty.
\end{align*}
We thus obtain
\begin{equation*}
u_1, u_2, u_3 \sim Cr^{-1} \sim Ce^{a_1 \eta} \sim C(t_{\max}-t)^{-1/2}\quad \text{ as }\quad \eta \to \infty
\end{equation*}
and the trajectories blow up with the rate $(t_{\max}-t)^{1/2}$, which indicates that the blow-up is of type-I.

\begin{description}
\item[Case 2 : $a_1 > 2a > 0$.]
\end{description}
In this case, the quasi-homogeneous component is
\begin{equation}
\label{Andrews2-blowup-lin}
\frac{dr}{d\eta} = -2ar,\quad
\frac{dw}{d\eta} = (2a - a_1)w
\end{equation}
and it is topologically conjugate to (\ref{Andrews2-blowup}) in a neighborhood of $(r,w) = (0,0)$ in $\{r,w\geq 0\}$.
The linearized system (\ref{Andrews2-blowup-lin}) is explicitly solved to obtain
\begin{equation*}
r(\eta) = r(0)e^{-2a\eta},\quad w(\eta) = w(0)e^{(2a-a_1)\eta}.
\end{equation*}
The maximal existence time in the original $t$-timescale is thus
\begin{align*}
t_{\max} &= \int_0^\infty \frac{dt}{d\tau}d\tau = \int_0^\infty \frac{dt}{d\tau}\frac{d\tau}{d\eta} d\eta\\
	&= \int_0^\infty r(\eta)^2 w(\eta) d\eta = C\int_0^\infty e^{(-2a - a_1)\eta}(1+o(1)) d\eta < \infty.
\end{align*}
It thus holds that the solution with large initial data $u_1(0) > 0$, $u_2(0) = u_3(0) > 0$ blows up in finite time.
The asymptotic behavior is obtained as follows:
\begin{align*}
t_{\max}-t &\sim C\int_\eta^\infty e^{(-2a - a_1)\tilde \eta}(1+o(1)) d\tilde \eta \sim Ce^{(-2a - a_1) \eta}(1+o(1))\quad \text{ as }\quad  \eta\to \infty.
\end{align*}
Therefore we finally have
\begin{align*}
u_1 &= \frac{1}{r} \sim Ce^{2a\eta} = C(e^{(-2a - a_1) \eta})^{\frac{2a}{-2a - a_1}} \sim C(t_{\max}-t)^{\frac{-2a}{2a + a_1}},\\
u_2 = u_3 &= \frac{1}{rw} \sim Ce^{a_1\eta} = C(e^{(-2a - a_1) \eta})^{\frac{a_1}{-2a - a_1}} \sim C(t_{\max}-t)^{\frac{-a_1}{2a + a_1}}
\end{align*}
as $t\to t_{\max}$.
\begin{description}
\item[Case 3 : $a_1 = 2a > 0$.]
\end{description}
In this case, the equilibrium $(r,w)=(0,0)$ of (\ref{Andrews2-blowup}) is a non-hyperbolic equilibrium.
\par
Here introduce a nonlinear (analytic) change of coordinate given as
$(R, \bar w) := (wr, w)$.
The corresponding vector field in this new coordinate is
\begin{equation*}
\frac{dR}{d\eta} = -a_1R,\quad \frac{dw}{d\eta} = -a_1w^2.
\end{equation*}
The solution is
\begin{equation*}
R(\eta) = R(0)e^{-a_1\eta},\quad w(\eta) = (a_1\eta + w(0)^{-1})^{-1}.
\end{equation*}
The maximal existence time of corresponding solution in the original $t$-timescale is
\begin{align*}
t_{\max} &= \int_0^\infty r(\eta)^2 w(\eta) d\eta = \int_0^\infty R(\eta)^2 w(\eta)^{-1} d\eta \\
	&= C\int_0^\infty e^{-2a_1\eta}(a_1\eta + w(0)^{-1})(1+o(1)) d\eta < \infty.
\end{align*}

The asymptotic behavior is
\begin{equation*}
t_{\max} -t \sim Ce^{-2a_1\eta}\eta \quad \text{ as }\quad \eta \to \infty.
\end{equation*}
In particular,
\begin{equation*}
\log(t_{\max} -t)^{-1} \sim C\left(\eta - \log(\eta)\right) \sim C\eta \quad \text{ as }\quad \eta \to \infty.
\end{equation*}
Then we finally have
\begin{align*}
u_1 &= \frac{1}{r} = \frac{w}{R} \sim Ce^{a_1\eta}\eta^{-1} = C(e^{-2a_1\eta}\eta)^{-1/2}\eta^{-1/2} \sim C\left\{(t_{\max}-t)^{-1}(\log(t_{\max}-t)^{-1})^{-1}\right\}^{1/2},\\
u_2 = u_3 &= \frac{1}{rw} = \frac{1}{R} \sim Ce^{a_1\eta} = C(e^{-2a_1\eta}\eta)^{-1/2}\eta^{1/2} \sim C\left\{(t_{\max}-t)^{-1}\log(t_{\max}-t)^{-1}\right\}^{1/2}
\end{align*}
as $t\to t_{\max}$.
\par
\bigskip
Summarizing the above arguments, we have the following result.
\begin{thm}[cf. \cite{AIU, A2002}]
Consider (\ref{Andrews-2}) with (\ref{ass-Andrews-Mat}) and sufficiently large initial data $u_1(0) > 0$ and $u_2(0) = u_3(0) > 0$.
Then the solution blows up at a time $t = t_{\max} < \infty$ with the following asymptotic behavior.
\begin{enumerate}
\item If $2a > a_1 > a_2=a_3\equiv a > 0$, then $u_1(t), u_2(t), u_3(t)\sim C(t_{\max}-t)^{-1/2}$ as $t\to t_{\max}$. 
\item If $ a_1 > 2a > a_2=a_3\equiv a > 0$, then $u_1(t) \sim C(t_{\max}-t)^{\frac{-2a}{2a + a_1}}$ and $u_2(t) = u_3(t)\sim C(t_{\max}-t)^{\frac{-a_1}{2a + a_1}}$ as $t\to t_{\max}$. 
\item If $2a = a_1 > a_2=a_3\equiv a > 0$, then $u_1(t) \sim C\left\{(t_{\max}-t)^{-1}(\log(t_{\max}-t)^{-1})^{-1}\right\}^{1/2}$ and $u_2(t) = u_3(t)\sim C\left\{(t_{\max}-t)^{-1}\log(t_{\max}-t)^{-1}\right\}^{1/2}$ as $t\to t_{\max}$. 
\end{enumerate}
\end{thm}
The third statement actually corresponds to Fact \ref{fact-A2} with the assumption (\ref{ass-Andrews-Mat}).

\begin{rem}
This example exhibits very interesting features of blow-ups both from viewpoints of dynamical systems with compactifications and blow-up solutions themselves.
Firstly, typical applications of compactifications show that invariant sets on the horizon cannot be found in general.
We have thus applied the second compactification to completely capturing blow-up behavior of systems.
Secondly, our analysis also shows that the blow-up rate depends not only on order of polynomials, but also on coefficients of original vector fields.
This feature is considered to stem from the necessity of {\em successive compactifications} mentioned above.
\end{rem}

\begin{rem}
During the present studies, one obtains an observation of blow-up behavior exhibiting the rate different from type-I.
In all examples above, blow-up rates are {\em different from each other among components}, which cannot be determined by the type $\alpha$ of vector fields at a glance.
This aspect of results depends on the system and determined by explicit calculations of dynamics at infinity.
We also see this aspect in the following examples, which indicates that it will be common in a large class of finite-time singularities.
\end{rem}

%
%


\subsection{Quenching traveling wave solutions}
The next example is the heat equation possessing reaction term with negative fractional powers:
\begin{equation}
\label{heat}
u_t = u_{xx} + (1-u)^{-\alpha},\quad t>0,\ x\in \mathbb{R},
\end{equation}
where $\alpha \geq 1$.

\begin{dfn}\rm
We say that a solution $u(t,x)$ of (\ref{heat}) {\em quenches} at a point $(T,x_0)$ if 
\begin{equation*}
\lim_{t\to T}u(t,x_0) = 0,\quad \lim_{t\to T}\frac{\partial u}{\partial t}(t,x_0) = -\infty.
\end{equation*}
\end{dfn}

\begin{fact}[\cite{K1975}]
Consider the initial-boundary value problem of (\ref{heat}) on $x\in (0,l)$ for some $l>0$ instead of $x\in \mathbb{R}$, with
\begin{equation*}
u(t,0) = u(t,l) = 0,\quad t>0,\quad \quad u(0,x) = 0,\quad x\in (0,l).
\end{equation*}
If $l > 2\sqrt{2}$, then the solution quenches.
\end{fact}

We pay our attention to the traveling wave solution $u(t,x) = \phi(x-ct) \equiv \phi(\xi)$ with $c\in\mathbb{R}\setminus \{0\}$, which should satisfy the following ordinary differential equation:
\begin{equation}
\label{heat-tw}
-cu' = u'' + u^{-\alpha} \quad \Leftrightarrow \quad 
\begin{cases}
\phi' = \psi, & \\
\psi' = -c\psi - \phi^{-\alpha}. & 
\end{cases}
\end{equation}

Since the system (\ref{heat-tw}) is singular at $\phi = 0$, we desingularize the vector field via the time-scale desingularization
\begin{equation*}
\frac{ds}{d\xi} = \phi(\xi)^{-\alpha},
\end{equation*}
which yields
\begin{equation}
\label{heat-tw-desing}
\begin{cases}
\dot \phi = \phi^{\alpha}\psi, & \\
\dot \psi = -c\phi^{\alpha}\psi - 1, & 
\end{cases}\quad \dot {} = \frac{d}{ds}.
\end{equation}
Numerical simulations with Poincar\'{e}-type compactifications (e.g., \cite{Mat}) implies that several initial data induce solutions of (\ref{heat-tw}) such that $\psi \to -\infty$ as $\xi\to \xi_{\max} < \infty$.
With this observation in mind, we apply the following directional compactification
\begin{equation*}
\phi = \frac{x}{\lambda},\quad \psi = -\frac{1}{\lambda} 
\end{equation*}
to the vector field for $(\lambda, x)$ associated with (\ref{heat-tw-desing}):
\begin{equation*}
\begin{cases}
\dot \lambda = cx^\alpha \lambda^{1-\alpha} - \lambda^2, & \\
\dot \psi = x(cx^\alpha \lambda^{-\alpha} - \lambda) - x^\alpha \lambda^{-\alpha}. & 
\end{cases}
\end{equation*}
Using the time-scale desingularization
\begin{equation*}
\frac{d\tau}{ds} = \lambda(s)^{-\alpha},
\end{equation*}
we have the desingularized vector field
\begin{equation}
\label{quench-desing1}
\begin{cases}
\frac{d \lambda}{d\tau} = cx^\alpha \lambda - \lambda^{2+\alpha}, & \\
\frac{d x}{d\tau} = x(cx^\alpha - \lambda^{1+\alpha}) - x^\alpha. & 
\end{cases}
\end{equation}

\subsubsection{The case $\alpha > 1$}
Equilibria at infinity should satisfy $\lambda = 0$.
We thus obtain the following equilibria at infinity: $(\lambda, x) = (0,0)\equiv p_0, (0,c^{-1})\equiv p_c$.
The Jacobian matrices at these equilibria are
\begin{align*}
J(\lambda, x) &= \begin{pmatrix}
cx^\alpha - (2+\alpha)\lambda^{1+\alpha} & \alpha c \lambda x^{\alpha-1}\\
-(1+\alpha)x\lambda^\alpha & (\alpha+1) c x^\alpha - \lambda^{1+\alpha} - \alpha x^{\alpha-1}
\end{pmatrix}\\
&= \begin{cases}
\begin{pmatrix}
0 & 0 \\
0 & 0
\end{pmatrix}
 & \text{ at $p_0$},\\
\begin{pmatrix}
c^{1-\alpha} & 0 \\
0 & c^{1-\alpha}
\end{pmatrix}
 & \text{ at $p_c$}.
\end{cases}
\end{align*}
When $c > 0$, then $p_c$ is source.
If $c < 0$, the value of $u$ corresponding to $p_c$ may attain negative value, which is not compatible in this case since $u$ is referred to as a temperature in a given environment.
Therefore, we focus on the equilibrium $p_0$.
\par
Since $p_0$ is non-hyperbolic, we desingularize $p_0$ via the following blow-up map:
\begin{equation*}
\Phi(r,\bar \lambda, \bar x) \equiv (\lambda, x),\quad \lambda := r^{\alpha-1}\bar \lambda,\  x := r^{\alpha+1}\bar x.
\end{equation*}
Restricting the transformed vector field of (\ref{heat-tw-desing}) via $\Phi$ to $\bar \lambda = 1$, we have
\begin{align*}
\frac{d r}{d\tau} &= \frac{r}{\alpha-1}(cr^{\alpha (\alpha+1)} \bar x^{\alpha} - r^{\alpha^2-1}),\\
\frac{d \bar x}{d\tau} &= -\frac{\alpha+1}{\alpha-1}\bar x (cr^{\alpha (\alpha+1)} \bar x^{\alpha} - r^{(\alpha-1)(\alpha+1)}) + cr^{\alpha (\alpha+1)}\bar x^{\alpha+1} - r^{\alpha^2 -1} \bar x - r^{\alpha^2 - 1}\bar x^{\alpha}.
\end{align*}
Using the time-scale desingularization
\begin{equation*}
\frac{d \eta}{d\tau} = r(\tau)^{\alpha^2-1},
\end{equation*}
we have
\begin{align}
\notag
\frac{d r}{d\eta} &= \frac{r}{\alpha-1}(cr^{\alpha+1} \bar x^{\alpha} - 1),\\
\label{heat-blowup-desing}
\frac{d \bar x}{d\eta} &= -\frac{\alpha+1}{\alpha-1}\bar x (cr^{\alpha+1} \bar x^{\alpha} - 1) + cr^{\alpha+1}\bar x^{\alpha+1} - \bar x - \bar x^{\alpha}.
\end{align}
The system (\ref{heat-blowup-desing}) has two equilibria whose images under $\Phi$ are $p_0$:
\begin{equation*}
\bar p_0 = (0,0)\quad \text{ and }\quad \bar p_\alpha = \left(0, \left(\frac{2}{\alpha-1}\right)^{1/(\alpha-1)} \right).
\end{equation*}
The Jacobian matrices of at these equilibria are
\begin{align*}
J(r, \bar x) &= \begin{pmatrix}
\frac{\alpha+2}{\alpha-1}cr^{\alpha+1}\bar x^{\alpha} - \frac{1}{\alpha-1} & \frac{\alpha}{\alpha-1}cr^{\alpha+2}\bar x^{\alpha-1}\\
-\frac{(\alpha+1)^2}{\alpha-1}cr^\alpha \bar x^{\alpha+1} + (\alpha+1)cr^\alpha \bar x^{\alpha+1} & -\frac{(\alpha+1)^2}{\alpha-1} cr^{\alpha+1} \bar x^\alpha + \frac{\alpha+1}{\alpha-1} + (\alpha+1)cr^{\alpha+1}\bar x^{\alpha} - 1 - \alpha \bar x^{\alpha-1}
\end{pmatrix}\\
&= \begin{cases}
\begin{pmatrix}
- \frac{1}{\alpha-1} & 0 \\
0 & \frac{2}{\alpha-1}
\end{pmatrix}
 & \text{ at $\bar p_0$},\\
\begin{pmatrix}
- \frac{1}{\alpha-1} & 0 \\
0 & -2
\end{pmatrix}
 & \text{ at $\bar p_\alpha$}.
\end{cases}
\end{align*}
These calculations show that $\bar p_0$ is a saddle and $\bar p_\alpha$ is a sink.
Trajectories in a small neighborhood of $r=0$ thus generically tend to a sink $\bar p_\alpha$, which have asymptotic expansions
\begin{equation*}
r(\eta) = Ce^{-\eta/(\alpha-1)}(1+ o(1)),\quad \bar x(\eta) - \left(\frac{2}{\alpha-1}\right)^{1/(\alpha-1)} = Ce^{-2\eta}(1+o(1)).
\end{equation*}
Note that time scales $\xi$ (original) and $\eta$ (transformed via $\Phi$) have the following relationship:
\begin{align*}
\frac{d\eta}{d\xi} &= \frac{ds}{d\xi}\frac{d\tau}{ds}\frac{d\eta}{d\tau} = \phi^{-\alpha} \lambda^{-\alpha} r^{\alpha^2-1}\\
	&=\left(\frac{r^{\alpha+1}\bar x}{r^{\alpha-1}}\right)^{-\alpha} (r^{\alpha-1})^{-\alpha} r^{\alpha^2-1}\\
	&= r^{-(\alpha+1)}\bar x^{-\alpha}\\
	&= Ce^{(\frac{\alpha+1}{\alpha-1} + 2\alpha)\eta}.
\end{align*}
Therefore the maximal existence time is finite since
\begin{equation*}
\xi_{\max} = C\int_0^\infty  e^{-(\frac{\alpha+1}{\alpha-1} + 2\alpha)\eta}(1+o(1))d\eta < \infty.
\end{equation*}
In particular, we have the following asymptotics:
\begin{equation*}
\xi_{\max} - \xi \sim Ce^{-(\frac{\alpha+1}{\alpha-1} + 2\alpha)\eta}\quad \text{ as }\eta\to \infty.
\end{equation*}
Finally, we have the following asymptotic behavior of $(\phi, \psi)$:
\begin{align*}
\phi(\xi) &= \frac{r^{\alpha+1}\bar x}{r^{\alpha-1}} = r^2 \bar x \sim Ce^{-\frac{2\alpha}{\alpha-1}\eta} \sim C(\xi_{\max}-\xi)^{\frac{2\alpha}{2\alpha^2-\alpha+1}}\\
	&\to 0\quad \text{ as }\quad\xi \to \xi_{\max},\\
\psi(\xi) &= -\frac{1}{r^{\alpha-1}} \sim -C(e^{-\eta/(\alpha-1)})^{-(\alpha-1)} = -Ce^\eta \sim -C(\xi_{\max}-\xi)^{\frac{1-\alpha}{2\alpha^2-\alpha+1}}\\
	&\to -\infty\quad \text{ as }\quad\xi \to \xi_{\max}.
\end{align*}

\begin{thm}
Assume that $\alpha > 1$ with $\alpha\in \mathbb{N}$ in (\ref{heat}).
The quenching traveling wave solutions for (\ref{heat}) is, if exists, characterized by trajectories whose initial data are on the stable manifold $W^s(\bar p_\alpha)$ for (\ref{heat-blowup-desing}).
The quenching rates, namely the extinction rate of $\phi$ and blow-up rate of $\psi$, are
\begin{equation*}
\phi(\xi) \sim C(\xi_{\max}-\xi)^{\frac{2\alpha}{2\alpha^2-\alpha+1}},\quad \psi(\xi) = -C(\xi_{\max}-\xi)^{\frac{1-\alpha}{2\alpha^2-\alpha+1}} \quad \text{ as }\quad \xi \to \xi_{\max}
\end{equation*}
with $C>0$.
\end{thm}

\subsubsection{The case $\alpha = 1$}
If $\alpha = 1$, then the scenario becomes completely different.
First, for the desingularized vector field (\ref{quench-desing1}) in the $\tau$-time scale, we have the following.
\begin{lem}
The desingularized vector field (\ref{quench-desing1}) with $\alpha = 1$ admit equilibria $(\lambda, x) = (0,0)\equiv p_0$, $(0,c^{-1})\equiv p_c$,
where the Jacobian matrix is
\begin{equation*}
J(\lambda, x) = \begin{pmatrix}
cx - 3\lambda^2 & c \lambda \\
-2x \lambda & 2c x - \lambda^2 - 1
\end{pmatrix}
= \begin{cases}
\begin{pmatrix}
0 & 0 \\
0 & -1
\end{pmatrix}
 & \text{ at $p_0$},\\
\begin{pmatrix}
1 & 0 \\
0 & 1
\end{pmatrix}
 & \text{ at $p_c$}.
\end{cases}
\end{equation*}
In other words, $p_0$ is semi-hyperbolic and $p_c$ is a source no matter what the value $c$ is.
\end{lem}
Then the only possible quenching solution would correspond to the trajectory on $W^{cs}(0,0)$ for (\ref{quench-desing1}).
Write the equation of consideration again:
\begin{equation}
\label{quench-desing2}
\begin{cases}
\frac{d \lambda}{d\tau} = cx \lambda - \lambda^3, & \\
\frac{d x}{d\tau} =  - x + x(cx - \lambda^2). & 
\end{cases}
\end{equation}
Now the Center Manifold Theorem guarantees the existence of a function $x = h(\lambda)$, whereas the concrete form seems to be unknown at a glance.
Obviously $(\lambda, x) = (0,0)$ is an equilibrium with the eigenvalues $(0,-1)$.
Therefore we compute the approximation of $h$ via the formal power series expansion $x = h(\lambda) = \sum_{n=2}^\infty \hat x_n \lambda^n$.
From (\ref{quench-desing2}) the invariant solution curve (namely, the center manifold) must satisfy
\begin{equation*}
(cx \lambda - \lambda^3)\frac{dh}{d\lambda} =  - x + x(cx - \lambda^2),
\end{equation*}
equivalently,
\begin{align*}
\left(c\sum_{n=2}^\infty \hat x_n \lambda^{n+1} - \lambda^3 \right) \sum_{n=2}^\infty n\hat x_n \lambda^{n-1} = -\sum_{n=2}^\infty \hat x_n \lambda^n + \sum_{n=2}^\infty \hat x_n \lambda^n\left(c\sum_{n=2}^\infty \hat x_n \lambda^n - \lambda^2\right)
\end{align*}
and
\begin{align*}
c\sum_{n=2}^\infty \sum_{k_1, k_2} \hat x_{k_1} k_2 \hat x_{k_2} \lambda^{2n} -\sum_{n=2}^\infty n\hat x_n \lambda^{n+2} = -\sum_{n=2}^\infty \hat x_n \lambda^n + c\sum_{n=2}^\infty \sum_{k_1, k_2} \hat x_{k_1} \hat x_{k_2} \lambda^{2n} -\sum_{n=2}^\infty \hat x_n \lambda^{n+2},
\end{align*}
where the sum $\sum_{k_1, k_2}$ runs over $k_1, k_2 \geq 2$ with $k_1 + k_2 = n$.
Comparison of coefficients for each $\lambda^n$ indicates that an approximate formal expansion of $x=h(\lambda)$ is given as $h= O(\lambda^5)$ as $\lambda\to 0$\footnote{
This result is obtained by calculations of coefficients up to $\lambda^4$.
Actually, higher order terms give little contributions to the present study.
}.
%
Substituting this formal series, we know that the asymptotic behavior on $W^c(0,0)$ is dominated by
\begin{equation*}
\frac{d \lambda}{d\tau} = - \lambda^3(1+o(1)),
\end{equation*}
which solves
\begin{equation*}
\lambda(\tau) = \frac{1}{\sqrt{2}}\tau^{-1/2} + o\left(\tau^{-1/2}\right).
\end{equation*}
Summarizing the arguments, the asymptotic behavior of solution $(\lambda(\tau), x(\tau))$ on $W^c(0,0)$ is expressed as
\begin{equation*}
\lambda(\tau) = \frac{1}{\sqrt{2}}\tau^{-1/2} + o\left(\tau^{-1/2}\right),\quad x(\tau)=Ce^{-\tau}(1+o(1)).
\end{equation*}
Recall that we have the following relationship of time (moving frame) scales:
\begin{equation*}
\frac{d\xi}{ds} \frac{ds}{d\tau} = \phi \lambda = \left(\frac{x}{\lambda}\right) \lambda = x.
\end{equation*}
The maximal existence time in $\xi$-scale is thus
\begin{equation*}
\xi_{\max} = C\int_0^\infty e^{-\tau}(1+o(1))d\tau < \infty,
\end{equation*}
thus the original solution $(\phi(\xi), \psi(\xi))$ admits finite-time singularity.
Note that
\begin{equation*}
\xi_{\max} - \xi = C\int_{\tau}^\infty e^{-\tilde \tau}(1+o(1))d\tilde \tau \sim Ce^{-\tau}\quad \text{ as }\quad \tau \to \infty,
\end{equation*}
equivalently,
\begin{equation*}
\tau \sim C\ln (\xi_{\max} - \xi)^{-1}\quad \text{ as }\quad \tau \to \infty.
\end{equation*}
Therefore we have
\begin{align*}
\phi(\xi) &= \frac{x}{\lambda} \sim C(\xi_{\max} - \xi)(\ln (\xi_{\max} - \xi)^{-1})^{1/2} \to 0\quad \text{ as }\quad \xi \to \xi_{\max},\\
\psi(\xi) &= -\frac{1}{\lambda} \sim -C(\ln (\xi_{\max} - \xi)^{-1})^{1/2} \to -\infty\quad \text{ as }\quad \xi \to \xi_{\max}.
\end{align*}

Summarizing the above arguments, we have the following theorem, which describes the existence of quenching traveling wave with concrete characterization of quenching rate.

\begin{thm}
Consider (\ref{heat}) with $\alpha = 1$.
Then the quenching traveling wave solutions for (\ref{heat}) is, if exists, characterized by trajectories whose initial data are on the center-stable manifold $W^{cs}(p_0)$ for (\ref{quench-desing2}).
The quenching rates; namely the extinction rate of $\phi$ and blow-up rate of $\psi$ are
\begin{equation*}
\phi(\xi) \sim C(\xi_{\max} - \xi)(\ln (\xi_{\max} - \xi)^{-1})^{1/2},\quad \psi(\xi) \sim -C(\ln (\xi_{\max} - \xi)^{-1})^{1/2} \quad \text{ as }\quad \xi \to \xi_{\max}
\end{equation*} 
with $C>0$.
\end{thm}


\subsection{FitzHugh-Nagumo system with nonlinear diffusion : extinction}
\label{section-ex-extinction}
Here we change our concern to other finite-time singularities such as finite-time extinction.
The present example is traveling wave solution for the following degenerate parabolic partial differential equation:
\begin{align}
\label{deg-para}
u_t &= d(u^{m+1})_{xx} + f_p(u),\quad (t,x)\in Q := (0,\infty)\times \mathbb{R},\\
\label{deg-para-bc}
u(0,x) &= u_0(x),\quad x\in \mathbb{R},\ 0\leq u_0 \leq 1,
\end{align}
where $m > 0$, $d = (m+1)^{-1}$, $f_p(u) = u^p(1-u)(u-a)$ with $a\in (0, 1/2)$, and $u_0\in C^0(\mathbb{R})$.
Assume that $m+p\equiv N \geq 1$.
Preceding works (e.g., \cite{LPB2004}) show that (\ref{deg-para}) admits a traveling wave solution which is identically zero on a subset of $\mathbb{R}$.

\begin{dfn}[Finite traveling wave, e.g., \cite{LPB2004}]\rm
A {\em finite traveling wave solution} of (\ref{deg-para})-(\ref{deg-para-bc}) is a solution of the form $u(t,x) = \varphi(x-ct)$ with a velocity $c\in \mathbb{R}$ satisfying $\varphi(z) \equiv 0$ for $z\geq w$ (or $z\leq w$) for some $w\in \mathbb{R}$. 
\end{dfn}
\par
Under the traveling wave profile ansatz $u(t,x) \equiv \phi(x-ct)\equiv \phi(\xi)$, the traveling wave equation associated with (\ref{deg-para}) becomes
\begin{equation*}
-c\phi' = (\phi^m \phi')' + f_p(\phi),\quad {}' = \frac{d}{d\xi},
\end{equation*}
equivalently,
\begin{equation}
\label{deg-para-tw}
\phi^m\phi' = \psi,\quad \phi^m \psi' = -c\psi - \phi^m f_p(\phi). 
\end{equation}
Observe that the system (\ref{deg-para-tw}) has a singularity at $\phi=0$, which induces a finite-time (frame) singularity for traveling wave solutions.
We introduce a time-variable desingularization
\begin{equation*}
\frac{d\xi}{d\tau} = \phi(\xi)^m
\end{equation*}
to obtain the following desingularized vector field:
\begin{equation}
\label{deg-para-tw-desing}
\dot \phi = \psi,\quad \dot \psi = -c\psi - f_{m+p}(\phi),\quad \dot {} = \frac{d}{d\tau}. 
\end{equation}
The system has equilibria at $(\phi, \psi) = (0,0)$ and $(1,0)$ for any $c\in \mathbb{R}$ and $a\in (0,1/2)$.
Also note that, for given $a\in (0,1/2)$, there is a connecting orbit from $(0,0)$ to $(1,0)$ with some positive value $c > 0$.
The Jacobian matrices at these equilibria are
\begin{equation*}
\begin{pmatrix}
0 & 1\\
(m+p) a\phi^{m+p-1} & -c
\end{pmatrix}
\text{ at }(0,0),\quad 
\begin{pmatrix}
0 & 1\\
1 - a & -c
\end{pmatrix}
\text{ at }(1,0),
\end{equation*}
respectively.
Eigenvalues are
\begin{equation*}
\left\{\frac{1}{2}(-c\pm \sqrt{c^2 + 4(m+p)a\phi^{m+p-1}}) \mid \phi = 0\right\} \text{ at }(0,0),\quad 
\left\{ \frac{-c\pm \sqrt{c^2 +4(1-a) } }{2} \right\}\text{ at }(1,0),
\end{equation*}
which imply that the point $(1,0)$ is always saddle, whereas $(0,0)$ is not hyperbolic when $m+p > 1$.

\subsubsection{The case $m+p=1$}
If $m+p=1$, then the origin $(0,0)$ is also saddle and hence we easily obtain an asymptotic behavior 
\begin{equation*}
\phi(\tau) = Ce^{\lambda_{\min} \tau}(1+o(1)) \text{ as }\tau\to -\infty,\text{ where }\lambda_{\min} > 0.
\end{equation*}
Therefore
\begin{equation*}
\xi_{\min} = -\int_{-\infty}^0 \phi(\tau)^m d\tau \sim - C\int_{-\infty}^0 e^{\lambda_{\min} m\tau} d\tau < \infty.
\end{equation*}
Therefore the trajectory $\{\phi(\tau)\}$ is defined on $(\xi_{\min}, +\infty)$ in $\xi$-scale and the function
\begin{equation*}
\varphi(\xi) \equiv  \begin{cases}
0 & \xi \in (-\infty, \xi_{\min}]\\
\phi(\xi) & \xi \in (\xi_{\min}, +\infty)
\end{cases}
\end{equation*}
is a finite traveling wave solution of (\ref{deg-para})
Moreover, we also have the asymptotic behavior
\begin{equation*}
\xi- \xi_{\min} \sim Ce^{\lambda_{\min} m \tau} \text{ as }\tau \to -\infty,
\end{equation*}
equivalently,
\begin{equation*}
e^{\lambda_{\min} \tau} \sim C(\xi - \xi_{\min})^{1/m} = C(\xi - \xi_{\min})^{1/(1-p)},
\end{equation*}
which coincides with the result in \cite{LPB2004}.
Therefore $\phi(\xi) \sim C(\xi - \xi_{\min})^{1/(1-p)}$ as $\xi \to \xi_{\min}+0$.
Existence of such finite traveling waves are discussed independently in \cite{Mat4}, including {\em compacton} traveling waves discussed in Section \ref{section-compacton}.

\subsubsection{The case $m+p>1$}
If $m+p \equiv N>1$, then the origin $(0,0)$ is not hyperbolic and hence we need an asymptotic behavior on the center(-unstable) manifold $W^c(0,0)$. 
According to center manifold theory (\cite{C1981}, see also \cite{LPB2004}), the center-unstable manifold of $(0,0)$ is locally given by
\begin{equation*}
\psi = h(\phi) \equiv C\phi^N + o(\phi^N)\quad \text{ as }\phi \to 0+,
\end{equation*}
where $C>0$ is a constant.
In this case, from (\ref{deg-para-tw-desing}) we have the vector field $\dot \phi = C\phi^N + o(\phi^N)$ on $W^c(0,0)$.
Further desingularizing this vector field by
\begin{equation*}
\frac{d\tau}{d\eta} = \phi^{N-1},
\end{equation*}
we finally have 
\begin{equation}
\label{deg-desing-eta}
\frac{d\phi}{d\eta} = C\phi + o(\phi).
\end{equation}
The origin $\phi = 0$ is the {\em hyperbolic source for (\ref{deg-desing-eta})}.
Note that the new system (\ref{deg-desing-eta}) is orbitally equivalent to $\dot \phi = C\phi^N + o(\phi^N)$ as long as $\phi > 0$.
\par
Dynamics around $\phi = 0$ for (\ref{deg-desing-eta}) is described by $\phi(\eta) = \phi(0)e^{C\eta}(1+o(1))$ near $\phi = 0$.
In the original frame coordinate, we have
\begin{align*}
\xi_{\min} &= \int_{-\infty}^0 \phi(\eta)^m \frac{d\tau}{d\eta} d\eta = \int_{-\infty}^0 \phi(\eta)^{1-N+m} d\eta = \int_{-\infty}^0 \phi(\eta)^{1-p} d\eta = \int_{-\infty}^0 e^{C(1-p)\eta} d\eta < \infty,
\end{align*}
which yields that
\begin{equation*}
\xi - \xi_{\min} \sim C'e^{C(1-p)\eta}(1+o(1))\quad \Leftrightarrow \quad \phi(\xi) \sim C(\xi - \xi_{\min})^{1/(1-p)}\quad \text{ as }\quad \xi \to \xi_{\min}+0,
\end{equation*}
which is the same as the case $N=1$.
There is an interesting fact that the practical regularity of traveling wave solutions depends only on $p$, independent of the degeneracy exponent $m$.

\begin{rem}
The result shown here is previously revealed in \cite{LPB2004}.
Our arguments here show that the treatment of asymptotic behavior near finite-time singularities (degenerate points or blow-up directions) can be identical in good agreements with previous studies of such asymptotic behavior.
\end{rem}

\subsection{The KdV equation with nonlinear dispersion : compacton}
\label{section-compacton}
{\em Compactons}, in the present interests in this subsection, are introduced by Rosenau and Hyman in the study of nonlinear dispersion in the formation of patterns in liquid drops \cite{RH1993}.
Typical model is the family of fully nonlinear Korteweg-de Vries (KdV) equations $K(m,n)$:
\begin{equation}
\label{gKdV}
u_t + (u^m)_x + (u^n)_{xxx} = 0,\quad m>0,\ 1\leq n\leq 3.
\end{equation}
The equation $K(2,1)$ is the well-known KdV equation and $K(3,1)$ is the modified KdV (mKdV) equation.
The essence for generating compacton is the case $n>1$.

\par
We pay attention to traveling wave solutions $u(t,x) = \phi(\xi)$ with $\xi = x-ct$ such that $\lim_{\xi \to -\infty} \phi(\xi) = 0$.
After integration of the original equation in $\xi$, we have the following equation for $\phi$:
\begin{equation}
-c\phi + \phi^m + (\phi^n)_{\xi \xi} = 0
\end{equation}
equivalently,
\begin{equation}
\label{sys-gKdV}
\begin{cases}
n\phi^{n-1}\phi' = \psi & \\
\psi' =  c\phi - \phi^m &
\end{cases},\quad {}' = \frac{d}{d\xi}.
\end{equation}
The equation (\ref{sys-gKdV}) has the singularity at $\phi = 0$ if $n>1$.
Introducing the time-scale transformation
\begin{equation*}
\frac{d\xi}{dz} = n\phi(\xi)^{n-1},
\end{equation*}
we have the desingularized vector field
\begin{equation}
\label{sys-gKdV-desing}
\begin{cases}
\dot \phi = \psi & \\
\dot \psi = n\phi^n (c - \phi^{m-1})
\end{cases},\quad \dot {} = \frac{d}{dz}.
\end{equation}
As the correspondence between soliton solution for $K(2,1)$ and homoclinic trajectories of the origin for (\ref{sys-gKdV-desing}) with $(m,n)=(2,1)$,  compactons are considered as homoclinic trajectories of the origin for (\ref{sys-gKdV-desing}) with $n>1$.

\par
Now we assume $n>1$ and $m+n-2 > 0$.
The Jacobian matrix for (\ref{sys-gKdV-desing}) is then
\begin{equation*}
\begin{pmatrix}
0 & 1 \\
n(nc\phi^{n-1} - (m+n-1)\phi^{m+n-2}) & 0
\end{pmatrix}.
\end{equation*}
The origin is therefore generically (i.e., as long as $c\not = 0$) {\em a nilpotent singularity}.
We pay attention to the case $c>0$.

\subsubsection{The case $m > 1$}
In this case, the Newton diagram for (\ref{sys-gKdV-desing}) is characterized as
\begin{equation*}
\gamma = \{\lambda p + (1-\lambda)q \mid \lambda \in [0,1], p = (0,2), q=(n+1,0)\}.
\end{equation*}
Therefore we choose the blow-up $\Phi : [0,\infty)\times \mathbb{R}^2\to \mathbb{R}^2$ as
\begin{equation*}
\Phi(\bar \phi, \bar \psi, r) = (\phi, \psi),\quad \phi = r^2\bar \phi,\ \psi = r^{n+1}\bar \psi.
\end{equation*}
Consider the chart $\{\bar \phi = +1\}$; namely $\phi = r^2,\ \psi = r^{n+1}\bar \psi$
Then the vector field (\ref{sys-gKdV-desing}) in this chart is
\begin{equation*}
\begin{cases}
\dot r = \frac{1}{2}r^n \bar \psi, &\\
\dot {\bar \psi} = -\frac{n+1}{2} r^{n-1}\bar\psi^2 + nc r^{n-1} - nr^{2m+n-3}.
\end{cases}
\end{equation*}
Now introducing the time-variable transform
\begin{equation*}
\frac{dz}{ds} = r^{-(n-1)},
\end{equation*}
we have the vector field
\begin{equation}
\label{sys-gKdV-desing-not-0}
\begin{cases}
\dot r = \frac{1}{2}r \bar \psi, &\\
\dot {\bar \psi} = -\frac{n+1}{2} \bar\psi^2 + nc - nr^{2m-2}.
\end{cases}
\end{equation}
Equilibria in the invariant manifold $\{r=0\}$ are
\begin{equation*}
p_\pm = \left(0, \pm \sqrt{\frac{2nc}{n+1}}\right).
\end{equation*}

The Jacobian matrices at these equilibria are
\begin{equation*}
\begin{pmatrix}
\pm \sqrt{\frac{nc}{2(n+1)}} & 0 \\
0 & \mp \sqrt{2nc(n+1)}
\end{pmatrix}
\end{equation*}
assuming $m>3/2$.
\par
Now assume that there is a connecting orbit from $p_+$ to $p_-$.
Then the asymptotic behavior of the orbit around $p_-$ for $r$-variable is characterized as
\begin{equation}
\label{asym-gKdV+}
r\sim e^{\mu_{\min} s}(1+o(1)) \quad \text{ as }s\to +\infty
\end{equation}
with some $\mu_{\min} < 0$.
Similarly the asymptotic behavior of the orbit around $p_+$ for $r$-variable is characterized as
\begin{equation}
\label{asym-gKdV-}
r\sim e^{\lambda_{\min} s}(1+o(1)) \quad \text{ as }s\to -\infty
\end{equation}
with some $\lambda_{\min} > 0$.
Thus the upper and lower bounds of the wave corresponding to the connecting orbit in $\xi$-frame scale are
\begin{align*}
\xi_{\max} &= n\int_0^\infty \phi(z)^{n-1} dz = n\int_0^\infty r(s)^{2(n-1)}\cdot r(s)^{-(n-1)} ds\\
	&\sim C\int_0^\infty e^{\mu_{\min} (n-1)s}(1+o(1)) ds < \infty,\\
\xi_{\min} &= -n\int_{-\infty}^0 \phi(z)^{n-1} dz = -n\int_{-\infty}^0 r(s)^{2(n-1)}\cdot r(s)^{-(n-1)} ds\\
	&\sim -C\int_{-\infty}^0 e^{\lambda_{\min} (n-1)s}(1+o(1)) ds > -\infty,
\end{align*}
which indicate that, if (\ref{sys-gKdV-desing-not-0}) admits a connecting orbit $\{(r(s), \bar \psi(s))\}_{s\in \mathbb{R}}$ from $p_+$ to $p_-$, the function
\begin{equation}
\label{sol-gKdV}
u(t,x) := \begin{cases}
r(\xi)^2 & \text{ if $\xi \equiv x-ct \in (\xi_{\min}, \xi_{\max})$} \\
0 & \text{ otherwise }
\end{cases}
\end{equation}
is a solution of (\ref{gKdV}).
Moreover, the asymptotic behavior (\ref{asym-gKdV+}) - (\ref{asym-gKdV-}) also implies that
\begin{equation*}
r(\xi) \sim \begin{cases}
C(\xi_{\max}- \xi)^{\frac{1}{n-1}} & \text{ as }\xi \to \xi_{\max}-0,\\
C(\xi - \xi_{\min})^{\frac{1}{n-1}} & \text{ as }\xi \to \xi_{\min}+0.
\end{cases}
\end{equation*}
The concrete form (\ref{sol-gKdV}) indicates that
\begin{equation*}
u(t,x) \sim \begin{cases}
C(\xi_{\max}- \xi)^{\frac{2}{n-1}} & \text{ as }\xi\equiv x-ct \to \xi_{\max}-0,\\
C(\xi - \xi_{\min})^{\frac{2}{n-1}} & \text{ as }\xi\equiv x-ct \to \xi_{\min}+0.
\end{cases}
\end{equation*}
which completely corresponds to the asymptotic behavior derived in \cite{R2005}.
In particular, the compacton traveling wave solution belong to 
\begin{equation}
\label{func-gKdV-case1}
\begin{cases}
C^{r,\sigma}(\mathbb{R}) & \text{ if $n-1 < 2$, where $r = \lfloor 2/(n-1)\rfloor$ and $\sigma = 2/(n-1) - \lfloor 2/(n-1)\rfloor$,}\\
C^{\sigma}(\mathbb{R}) & \text{ if $n-1 > 2$, where $\sigma = 2/(n-1)$.}
\end{cases}
\end{equation}
Note that this compacton wave solution in {\em not} even $C^1$ if $n>2$.

\subsection{Periodic blow-up beyond type-I blow-up rate and grow-up}
The final example is a blow-up solution associated with nontrivial invariant sets at infinity.
Our present concern is periodic blow-up whose blow-up rate is not of type-I.
Fixing $m,n\in \mathbb{Z}_{\geq 1}$, 
consider the Li\'{e}nard equation (e.g., \cite{DH1999})
\begin{equation}
\label{Lienard}
\begin{cases}
x' = y, & \\
y' = -(\epsilon x^m + \sum_{k=0}^{m-1} a_k x^k) - y(x^n + \sum_{k=0}^{n-1} b_k x^k), &
\end{cases}\quad {}' = \frac{d}{dt}
\end{equation}
with $\epsilon \in\{\pm 1\}$ if $m\not = 2n+1$, and $\epsilon \in \mathbb{R}\setminus \{0\}$ if $m=2n+1$.

Dumortier and Herssens discuss the asymptotic behavior of solutions for (\ref{Lienard}) at infinity in \cite{DH1999}.
In particular, the special choice of $(m,n)$ yields the following behavior.

\begin{itemize}
\item If $m=2n+1$ with {\em even} $n$ and $\epsilon > (4(n+1))^{-1}$, then (\ref{Lienard}) admits a repelling periodic orbit at infinity. In particular, this periodic orbit is of saddle type.
\item If $m=2n+1$ with {\em odd} $n$ and $\epsilon > (4(n+1))^{-1}$, then (\ref{Lienard}) admits a non-hyperbolic periodic orbit.
\end{itemize}

Periodic blow-up theorem (Proposition \ref{prop-periodic-blowup}) indicates that the system (\ref{Lienard}) with backward time direction admits a periodic blow-up solution with blow-up rate $O(t_{\max}-t)^{-1/n}$; namely, it is type-I blow-up\footnote{
We easily observe that (\ref{Lienard}) is asymptotically quasi-homogeneous with type $(1,n+1)$ and order $n+1$.
See \cite{Mat}.
}, {\em provided that $n$ is an even integer}.
Our interest here is then the asymptotic behavior of periodic orbits at infinity with the following setting.

\begin{ass}
\label{ass-Lienard}
Let $m =2n+1$ and $n$ be an odd integer.
Moreover, set $\epsilon = 1$.
Finally, set the sequence $\{a_k\}_{k=0}^{2n}$ and $\{b_k\}_{k=0}^{n-1}$ as
\begin{equation*}
a_k = 0\ (k=0,\cdots, 2n-1),\ a_{2n} = 1,\quad b_k = 0\ (k=0,\cdots, n-1).
\end{equation*}
Namely, our system is reduced to
\begin{equation}
\label{Lienard-ours}
\begin{cases}
x' = y, & \\
y' = -(x^{2n+1} + x^{2n}) - x^n y, &
\end{cases}\quad {}' = \frac{d}{dt}.
\end{equation}
\end{ass}
The most essential point in the assumption is that {\em $n$ is odd}.
Introducing the quasi-polar coordinate compactification
\begin{equation}
\label{cpt-qpolar}
x = \frac{{\rm Cs}\theta}{r},\quad y = \frac{{\rm Sn}\theta}{r^{n+1}}
\end{equation}
the desingularized vector field for (\ref{Lienard}) is
\begin{equation}
\begin{cases}
\dot r = r {\rm Cs}^n \theta {\rm Sn}^2\theta +r^2 {\rm Cs}^{2n} \theta {\rm Sn}\theta, & \\
\dot \theta = -\left(1+ {\rm Cs}^{n+1} \theta {\rm Sn}\theta + r{\rm Cs}^{2n+1} \theta \right),&
\end{cases}\quad \dot{} = \frac{d}{d\tau},
\end{equation}
where
\begin{equation*}
\frac{d\tau}{dt} = r^{-n}.
\end{equation*}
The quasi-trigonometric functions ${\rm Cs}$ and ${\rm Sn}$ ($(1,l)$-quasi-trigonometric functions) are analytic functions given by the solutions of the following Cauchy problem (e,.g., \cite{DH1999}):
\begin{equation*}
\frac{d}{d\theta}{\rm Cs}\theta = -{\rm Sn}\theta,\quad \frac{d}{d\theta}{\rm Sn}\theta = {\rm Cs}^{2l-1}\theta,\quad 
\begin{cases}
{\rm Cs}0 = 1 &\\
{\rm Sn}0 = 0 &
\end{cases}.
\end{equation*}
These functions satisfy 
\begin{equation}
\label{trigonometric}
{\rm Cs}^{2l}\theta + l {\rm Sn}^{2}\theta = 1\quad \text{ for all } \theta,
\end{equation}
and both ${\rm Cs}\theta$ and ${\rm Sn}\theta$ are $T$-periodic with
\begin{equation*}
T = T_{1,l} = \frac{2}{\sqrt{l}} \int_0^1 (1-t)^{-1/2}t^{(1-2l)/2l}dt.
\end{equation*}
Functions ${\rm Cs}$ and ${\rm Sn}$ satisfy
\begin{align*}
&{\rm Cs}(-\theta) = {\rm Cs}\theta,\quad {\rm Sn}(-\theta) = -{\rm Sn}\theta,\\
&{\rm Cs}(T/2 - \theta) = -{\rm Cs}\theta,\quad {\rm Sn}(T/2 - \theta) = {\rm Sn}\theta,\\
&{\rm Cs}(T/2 + \theta) = -{\rm Cs}\theta,\quad {\rm Sn}(T/2 + \theta) = -{\rm Sn}\theta.
\end{align*}
We immediately know that $r=0$ is invariant and $\dot \theta < 0$ for sufficiently small $r$.
Therefore it is useful to consider the solution $r$ as a function of $\theta$ followed by the vector field (cf. \cite{DH1999})
\begin{align}
\notag
\frac{dr}{d\theta} &= -\frac{r {\rm Cs}^n \theta {\rm Sn}^2\theta + r^2 {\rm Cs}^{2n} \theta {\rm Sn}\theta}{\left(1+ {\rm Cs}^{n+1} \theta {\rm Sn}\theta \right) + r{\rm Cs}^{2n+1} \theta}\\
\label{dr/d_theta}
	&\equiv -\frac{\gamma_1(\theta) r + \gamma_2(\theta)r^2 }{1+ {\rm Cs}^{n+1} \theta {\rm Sn}\theta } + O(r^3)
\end{align}
with
\begin{equation*}
\gamma_1(\theta) = {\rm Sn}^2\theta {\rm Cs}^n \theta,\quad 
\gamma_2(\theta) = {\rm Sn}\theta {\rm Cs}^{2n} \theta - \frac{\gamma_1(\theta) {\rm Cs}^{2n+1}\theta}{1 + {\rm Sn}\theta {\rm Cs}^{n+1} \theta}.
\end{equation*}
First we know the following property.

\begin{prop}
\label{prop-Lienard-nonhyp-per}
Under Assumption \ref{ass-Lienard}, the Lienard equation (\ref{Lienard}) admits a periodic orbit at infinity, which is non-hyperbolic and attracting.
Similarly, if we replace the assumption $a_{2n}=1$ by $a_{2n}=-1$, the corresponding periodic orbit is non-hyperbolic and repelling.
\end{prop}

\begin{proof}
See Appendix \ref{proof-Lienard-nonhyp-per}.
A detailed derivation of (\ref{dr/d_theta}) is also shown there.
\end{proof}
The proof shows that the most essential point comes from the property of the integral 
\begin{equation*}
\alpha(\theta) = -\int_0^\theta \frac{\gamma_1(\psi) d\psi}{1 + {\rm Sn}\psi {\rm Cs}^{n+1}\psi}.
\end{equation*}
It is useful to introduce $\phi = -\theta$ in the following arguments, as the proof of Proposition \ref{prop-Lienard-nonhyp-per}.
We summarize the facts about $\alpha(\phi)$ derived from the above arguments.
Details are summarized in Appendix \ref{proof-Lienard-nonhyp-per} for details.
\begin{lem}
Using the new angular component $\phi \equiv -\theta$, the following properties hold\footnote{
The same argument as the proof shows that $\alpha(T) \not = 0$ if $n$ is even. See \cite{DH1999} for details.
}.
\begin{itemize}
\item $\alpha = \alpha(\theta)$ is smooth and $T$-periodic. In particular, $e^{\alpha(\phi)} > 0$ holds for all $\phi\in \mathbb{R}$.
\item $\alpha(\phi) < 0$ for $\phi\in (0,T/2)$, $\alpha(T/2) = 0$ and $\alpha(\phi) > 0$ for $\phi\in (T/2,T)$.
\item $\alpha(T) =0$.
\end{itemize}
\end{lem}

Now we move to study the asymptotic behavior of solutions near periodic orbit at infinity.
The property $\alpha(T)=0$ indicates that the leading term in (\ref{dr/d_theta}) does not affect the stability of periodic orbits.
We thus introduce a nonlinear transform $(R,\phi) = (h(r,\phi), \phi)$ for reducing the original problem to simpler one.
 
\begin{prop}
\label{prop-per-blowup-r2R}
Let $h(r,\phi)$ be a function given as
\begin{align}
\notag
h(r,\phi) &\equiv e^{\tilde \alpha(\phi)}r + \left( e^{2\tilde \alpha(\phi)} \int_0^\phi \frac{(e^{-\tilde \alpha(\psi)}-1)\gamma_2(\psi) d\psi}{1 - {\rm Sn}\psi {\rm Cs}^{n+1}\psi} \right) r^2,\\
\label{alpha}
\tilde \alpha(\phi) &= \int_0^\phi \frac{\gamma_1(\psi) d\psi}{1 - {\rm Sn}\psi {\rm Cs}^{n+1}\psi},\quad 
\end{align}
Then $h$ is smooth in $(r,\phi)$ and positive for $r>0$ and all $\phi$.
Moreover, the vector field (\ref{dr/d_phi}) in the new coordinate $(R,\phi) \equiv (h(r,\phi), \phi)$ is transformed smoothly into
\begin{equation}
\label{simpler-R}
\frac{dR}{d\phi} = -\frac{\gamma_2(\phi)R^2 }{1- {\rm Cs}^{n+1} \phi {\rm Sn}\phi } + O(R^3).
\end{equation}
\end{prop}

\begin{proof}
See Appendix \ref{proof-per-blowup-r2R}.
\end{proof}

We also derive the inverse formula for $r$ as a function of $R$, which is used later:
\begin{align}
\label{R2r}
r &= \frac{-b + \sqrt{b^2 + 4aR}}{2a}\quad \text{(since $r> 0$)},\\
\notag
a &= e^{2\tilde \alpha(\phi)} \int_0^\phi \frac{(e^{-\tilde \alpha(\psi)}-1)\gamma_2(\psi) d\psi}{1 - {\rm Sn}\psi {\rm Cs}^{n+1}\psi},\quad 
b = e^{\tilde \alpha(\phi)},
\end{align}
where $\tilde \alpha(\phi)$ is given in (\ref{alpha}).

\begin{rem}
The asymptotic formula for small $R$ is very useful later, which is given as follows by the Taylor formula $\sqrt{1+x} \sim 1+\frac{1}{2}x$ near $x=0$:
\begin{equation}
\label{R2r-asymptotic}
r \sim e^{-\tilde \alpha(\phi)}R
\quad \text{ as }R\to 0.
\end{equation}
Indeed, 
\begin{equation*}
r = \frac{-b + \sqrt{b^2 + 4aR}}{2a} \sim \frac{-b + b \left(1+\frac{2a}{b^2}R\right)}{2a} = \frac{R}{b} = e^{-\tilde \alpha(\phi)}R\text{ as }R\to 0.
\end{equation*}
Since $e^{-2\tilde \alpha (\phi)}$ is uniformly bounded and bounded away from $0$, the asymptotics of $r$ is determined by $R$ and vice versa. 
\par
\end{rem}
Consider {\em the principal part} of (\ref{simpler-R}):
\begin{equation}
\label{simpler-R-principal}
\frac{dR}{d\phi} = -\frac{\gamma_2(\phi)R^2 }{1- {\rm Cs}^{n+1} \phi {\rm Sn}\phi },
\end{equation}
which is directly solved with respect to $\phi$ to obtain
\begin{equation}
\label{R-formula-nonhyp}
R(\phi) = R_0 \left(R_0 \int_0^\phi \Gamma_2(\psi) d\psi + 1 \right)^{-1}\quad \text{ with }\quad R(0) = R_0 > 0.
\end{equation}
where $\Gamma_2(\psi) = \gamma_2(\psi) / (1-{\rm Cs}^{n+1}\psi {\rm Sn}\psi)$, which is also given in (\ref{Gamma_2}).
We use the following properties for obtaining the asymptotic behavior of $R$ in terms of $t$.

\begin{lem}
\label{lem-Gamma2}
The following properties hold.
\begin{itemize}
\item There is a positive constant $C_2>0$ such that $\Gamma_2(\psi) \leq C_2$ for all $\psi \geq 0$.
\item $\int_0^T \Gamma_2(\psi) d\psi \equiv \Gamma_T > 0$.
\item $\int_0^\phi \Gamma_2(\psi) d\psi = N \Gamma_T + \int_0^{\phi -NT} \Gamma_2(\psi) d\psi$, where $N$ is the integer such that $0\leq \phi -NT < T$.
We write such $N$ as $N_\phi$.
\item The integral $\int_0^\phi \Gamma_2(\psi) d\psi$ is positive for all $\phi \in [0,T]$.
In particular, the integral $\int_0^\phi \Gamma_2(\psi) d\psi$ is positive for all $\phi > 0$.
\item (Asymptotic behavior of $\int_0^\phi \Gamma_2(\psi) d\psi $) $\int_0^{\phi} \Gamma_2(\psi) d\psi \sim C_3 \phi$ holds for some positive constant $C_3>0$ as $\phi \to \infty$
\end{itemize}
\end{lem}

\begin{proof}
All statements except the last one immediately follow from the definition.
We shall prove the last statement.
Since $\int_0^\phi \Gamma_2(\psi) d\psi = N_\phi \Gamma_T + \int_0^{\phi -N_\phi T} \Gamma_2(\psi) d\psi$ holds for all $\phi > 0$, we have
\begin{equation*}
0\leq \frac{\int_0^{\phi -NT} \Gamma_2(\psi) d\psi}{N_\phi \Gamma_T} \leq \frac{C_2 T}{N_\phi \Gamma_T}
\to 0\quad \text{ as }\quad \phi\to \infty.
\end{equation*}
There is a constant $C_3 > 0$ such that $\Gamma_T = C_3 T$.
Since $\phi - N_\phi T \equiv \phi' \in [0, T)$ holds for all $\phi$, we further have
\begin{equation*}
\frac{C_3 \phi}{N_\phi \Gamma_T} = \frac{C_3 \phi}{N_\phi C_3 T} = \frac{N_\phi T + \phi'}{N_\phi T} = \left(1 + \frac{\phi'}{N_\phi T}\right) \to 1\quad \text{ as }\quad \phi\to \infty.
\end{equation*}
Therefore we have
\begin{align*}
\frac{\int_0^{\phi} \Gamma_2(\psi) d\psi - C_3 \phi}{C_3 \phi} &= \frac{N_\phi \Gamma_T + \int_0^{\phi -N_\phi T} \Gamma_2(\psi) d\psi - C_3 \phi}{C_3 \phi}\\
	&= \frac{N_\phi  \Gamma_T - C_3 \phi}{C_3 \phi} + \frac{ \int_0^{\phi -N_\phi T} \Gamma_2(\psi) d\psi }{C_3 \phi}\\
	&= \left\{ \left(1 + \frac{\phi'}{N_\phi T}\right)^{-1} - 1\right\} + \frac{ \int_0^{\phi -N_\phi T} \Gamma_2(\psi) d\psi }{C_3 \phi},
\end{align*}
which goes to $0$ as $\phi \to \infty$.
This convergence shows that $\int_0^{\phi} \Gamma_2(\psi) d\psi \sim C_3 \phi$ as $\phi \to \infty$.
\end{proof}

The above properties indicate
\begin{equation*}
\left(R_0 \int_0^\phi \Gamma_2(\psi) d\psi + 1 \right)^{-1} \leq \left(R_0 N_\phi \Gamma_T + 1 \right)^{-1},\quad \forall \phi > 0.
\end{equation*}

\begin{prop}[Asymptotic behavior of $R$]
\label{prop-asym-R}
The \lq\lq principal part" (\ref{simpler-R-principal}) of (\ref{simpler-R}) dominates the behavior of $R$ in the following sense.
If $R^{\rm original}$ and $R^{\rm pri}$ denote the solutions of (\ref{simpler-R}) and (\ref{simpler-R-principal}), respectively.
Then we have 
\begin{equation*}
R^{\rm original}(\phi) = R^{\rm pri}(C_3\phi + o(\phi))\quad \text{ as }\quad \phi\to \infty,
\end{equation*}
where $C_3$ is the positive constant obtained in Lemma \ref{lem-Gamma2}.
\end{prop}

\begin{proof}
See Appendix \ref{proof-asym-R}.
\end{proof}

Going back to the original problem with original $t$-timescale, we have
\begin{align*}
t_{\max} &= \int_0^\infty \frac{dt}{d\tau}d\tau = \int_0^\infty \frac{dt}{d\tau}\frac{d\tau}{d\phi} d\phi 
	 \leq C\int_0^\infty \left(R_0 \int_0^{\phi} \Gamma_2(\psi) d\psi + 1 \right)^{-n} (1+o(1)) d\phi \\
	&\leq C\int_0^\infty \left(R_0 N_\phi \Gamma_T + 1 \right)^{-n} (1+o(1)) d\phi
	= C\sum_{N=0}^\infty \left(R_0 N \Gamma_T + 1 \right)^{-n} (1+o(1)),
\end{align*}
which is finite {\em if $n>1$}\footnote{
The first inequality follows from the upper estimate of $d\phi / d\tau$.
The term $o(1)$ in the rightmost hand side is the sense \lq\lq as $N\to \infty$"
}.
On the other hand,
\begin{align*}
t_{\max} &= \int_0^\infty \frac{dt}{d\tau}d\tau = \int_0^\infty \frac{dt}{d\tau}\frac{d\tau}{d\phi} d\phi 
	= \int_0^\infty \frac{r^n}{1-{\rm Cs}^{n+1}\phi {\rm Sn}\phi + r{\rm Cs}^{2n+1}\phi}  d\phi\\
	& \geq C\int_0^\infty \left(R_0 \int_0^{\phi} \Gamma_2(\psi) d\psi + 1 \right)^{-n} (1+o(1)) d\phi
	\geq C\int_0^\infty \left(R_0 C_2 \Gamma_T  \phi + 1 \right)^{-n} (1+o(1))d\phi,
\end{align*}
which diverges if $n=1$\footnote{
The first inequality follows from the lower estimate of $d\phi / d\tau$ along the solution $R^{\rm original}$.
}, where $C_2$ is a positive constant given in Lemma \ref{lem-Gamma2}.
The above calculations yield that the solution of (\ref{Lienard-ours}) with odd $n$ whose image via compactification converges to the periodic orbit on the horizon is a (finite-time) blow-up solution if $n\geq 3$.
On the other hand, if $n=1$, the corresponding solution is a {\em grow-up solution}, namely, the solution which diverges in infinite time.
Now we are ready to calculate the behavior of blow-up and grow-up solutions.
%
%
\subsubsection{Asymptotics of $t_{\max} - t$ with $n\geq 3$}
Assume first that $n\geq 3$. 
Then we have
\begin{align*}
t_{\max} - t &= C\int_\phi^\infty \left(R_0 \int_0^\eta \Gamma_2(\psi) d\psi + 1 \right)^{-n} (1+o(1))d\eta \\
	&=C \int_\phi^\infty \left(C_3R_0 \eta + 1 \right)^{-n}(1+o(1)) d\eta \sim C (\phi + c)^{-n+1}
\end{align*}
as $\phi\to \infty$ for positive constants $C', c$, where we have used the asymptotics of $\int_0^\phi \Gamma_2(\psi) d\psi$ shown in Lemma \ref{lem-Gamma2}.
In particular, we have
\begin{equation}
\label{asym-per-ngt3}
\phi \sim (t_{\max} - t)^{-1/(n-1)}\quad \text{ as }\quad \phi \to \infty\ (\Leftrightarrow t\to t_{\max})
\end{equation}
up to multiplication of constants.
Substituting this asymptotics into (\ref{R-formula-nonhyp}), we have
\begin{equation*}
R(\phi)^{-1} \sim C \left(C_3R_0 \phi + 1 \right) \sim C(t_{\max} - t)^{-1/(n-1)} \quad \text{ as }\phi \to \infty\ (\Leftrightarrow t\to t_{\max}),
\end{equation*}
which yields
\begin{equation*}
r^{-1} \in \Theta((t_{\max}-t)^{-1/(n-1)})\quad \text{ as }\quad t\to t_{\max},
\end{equation*}
where we have also used the asymptotics (\ref{R2r-asymptotic}).
Note that the present blow-up rate is strictly faster than type-I blow-up.
In particular, our study shows that {\em non-hyperbolic periodic orbits on the horizon can induce blow-up solutions whose blow-up rates are different from type-I}.
\par
Notice that we already have the asymptotic behavior of angular component $\phi$ as (\ref{asym-per-ngt3}).
Since $d\phi / d\tau$ is positive, bounded and bounded away from zero, we have
\begin{equation*}
\phi \equiv -\theta \sim -C(t_{\max} - t)^{-1/(n-1)}\quad  \text{ as }\quad \phi \to \infty\ (\Leftrightarrow t\to t_{\max}). 
\end{equation*}
This asymptotic behavior is different from that of type-I periodic blow-ups stated in Proposition \ref{prop-periodic-blowup}.

%
%
\subsubsection{Asymptotics of $t_{\max} - t$ with $n=1$}
Next assume that $n= 1$. 
Then we have
\begin{align*}
t &= \left(\int_0^{\phi_1} + \int_{\phi_1}^{\phi} \right)\left(R_0 \int_0^\eta \Gamma_2(\psi) d\psi + 1 \right)^{-1} (1+o(1))d\eta\\
	&= \int_0^{\phi_1} \left(R_0 \int_0^\eta \Gamma_2(\psi) d\psi + 1 \right)^{-1}  (1+o(1)) d\eta + \int_{\phi_1}^{\phi} \left(R_0 C_3 \eta + 1 \right)^{-1} (1+o(1)) d\eta \\
	&= \left\{ \int_0^{\phi_1} \left(R_0 \int_0^\eta \Gamma_2(\psi) d\psi + 1 \right)^{-1} d\eta + \frac{1}{R_0 C_3}\log \frac{R_0 C_3 \phi + 1}{R_0 C_3 \phi_1 + 1} \right\}(1+o(1)) \sim C\log \frac{\phi}{\phi_1}
\end{align*}
for $1\ll \phi_1 \ll \phi\to \infty$.
In particular, we have
\begin{equation*}
\phi \sim Ce^t\quad \text{ as }\quad \phi \to \infty\ (\Leftrightarrow t\to \infty). 
\end{equation*}
Therefore
\begin{equation*}
R(\phi)^{-1} = R_0 \left(R_0 \int_0^\phi \Gamma_2(\psi) d\psi + 1 \right) \sim C \phi \sim Ce^t \quad  \text{ as }\quad \phi \to \infty\ (\Leftrightarrow t\to \infty).
\end{equation*}

\subsubsection{The final result}

Summarizing all our arguments, we have obtained the following statement.

\begin{thm}
Consider (\ref{Lienard-ours}) with odd $n$.
Then, for sufficiently large initial data, the system (\ref{Lienard}) admits a periodic divergent solution $(x(t), y(t))$ such that, with generic positive constant $C>0$,
\begin{enumerate}
\item if $n\geq 3$, $(x(t), y(t))$ blows up at $t = t_{\max} < \infty$.
Moreover, the solution has the following blow-up rate:
\begin{equation*}
x(t) \in \Theta\left( \frac{{\rm Cs}(C(t_{\max}-t)^{-1/(n-1)})}{(t_{\max}-t)^{1/(n-1)}} \right),\quad y(t) \in \Theta \left( \frac{-{\rm Sn}(C(t_{\max}-t)^{-1/(n-1)})}{(t_{\max}-t)^{(n+1)/(n-1)}} \right)\quad \text{ as }t\to t_{\max}.
\end{equation*}
\item if $n = 1$, $(x(t), y(t))$ grows up.
Namely, the solution diverges as $t\to \infty$.
The asymptotic behavior is described as follows:
\begin{equation*}
x(t) \in \Theta\left( e^t {\rm Cs}(Ce^t) \right),\quad y(t)\in \Theta\left( e^{(n+1)t}{\rm Sn}(-Ce^t)\right)\quad \text{ as }t\to t_{\max}.
\end{equation*}
\end{enumerate}
\end{thm}

\section*{Conclusion}

In this paper, we have studied a universal mechanism of finite-time singularities in dynamical systems generated by ordinary differential equations from the geometric viewpoint.
Our concern contains blow-up solutions, finite-time extinctions, compacton traveling wave solutions and quenching solutions.
Our approach is based on compactifications (in case of blow-up solutions and quenching solutions) and precise descriptions of asymptotic behavior near finite-time singularities.
We have shown that, with the help of asymptotic behavior of trajectories near {\em hyperbolic} singularities, {\em smooth and orbital equivalence} of dynamical systems, asymptotic behavior near finite-time singularities can be described as trajectories on center-stable manifolds of corresponding singularities or invariant sets with rigorous rate of blow-ups, extinctions or quenching.
The approach is shown to work in many examples with the comprehensive mechanism of finite-time singularities depending on the form of vector fields, such as order of polynomials and coefficients.
In particular, in the case of blow-ups, {\em component-wise} rigorous blow-up rates not only with type-I rates but also faster and slower rates than type-I rates can be detected.
We believe that the present study will play a key role in revealing a universal mechanism of finite-time singularities for differential equations including partial differential equations.
\par
\bigskip
We end this paper addressing several further directions from the present study.
One of natural questions will be whether the preceding methodology is available to {\em infinite dimensional} dynamical systems for detecting blow-up rates of blow-up solutions for evolutionary equations including partial differential equations.
To this end, the infinite dimensional analogue of our treatments \cite{EG2006, Mat} including compactifications themselves are necessary.
Even if an infinite dimensional \lq\lq compactification" is developed, dynamics at infinity is intrinsically an infinite dimensional problem. 
For example, consider the nonlinear heat equation $u_t = -Au + f(u) = u_{xx} + u^p$ on $\mathbb{R}$ with some $p>1$.
It is well-known that several blow-up solutions with large initial data are governed by the corresponding ODE $u' = u^p$ by the help of comparison principle. 
In such a case, the diffusion effect works little for blow-up solutions and several spectral properties for the linear operator $A + f'(u)$ can be violated near infinity.
In other words, \lq\lq finite dimensional assumptions" for typical treatments of infinite dimensional dynamical systems (such as center manifolds \cite{C1981}) are not guaranteed in general.
\par
Our blow-up and extinction rates are obtained from asymptotic behavior of trajectories on center manifolds of invariant sets on the horizon for desingularized vector fields.
In the present arguments, we have actually obtained {\em the lowest order asymptotic expansion} of blow-up solutions as well as grow-up solutions, extinction and compactons.
Note that the higher order asymptotic expansion of solutions on center manifolds can be achieved by precise forms of center manifolds as graphs of smooth functions \cite{C1981, CAA1980}.
Using such precise information, there is a possibility that we obtain {\em higher order asymptotic expansions of blow-up solutions} near blow-up times as well as other finite-time singularities.

\section*{Acknowledgements}
The author was partially supported by Program for Promoting the reform of national universities (Kyushu University), Ministry of Education, Culture, Sports, Science and Technology (MEXT), Japan, World Premier International Research Center Initiative (WPI), MEXT, Japan, and JSPS Grant-in-Aid for Young Scientists (B) (No. 17K14235).
He would like also to thank Professors Koichi Anada, Tetsuya Ishiwata and Takeo Ushijima for giving him very essential suggestions to the present study.

\bibliographystyle{plain}
\bibliography{qh_blow_up_3}

\appendix
\section{Proofs of statements}


\subsection{Proof of Proposition \ref{prop-Lienard-nonhyp-per}}
\label{proof-Lienard-nonhyp-per}

It turns out that, for sufficiently small $r>0$, $\dot \theta$ is strictly negative.
We thus introduce $\phi = -\theta$ to obtain
\begin{equation*}
\begin{cases}
\dot r = r {\rm Cs}^n \phi {\rm Sn}^2\phi -r^2 {\rm Cs}^{2n} \phi {\rm Sn}\phi, & \\
\dot \phi = \left(1- {\rm Cs}^{n+1} \phi {\rm Sn}\phi + r{\rm Cs}^{2n+1} \phi \right). &
\end{cases}
\end{equation*}
Here we consider the solution $r$ as a function of $\phi$ followed by the vector field
\begin{align}
\notag
\frac{dr}{d\phi} &= \frac{r {\rm Cs}^n \phi {\rm Sn}^2\phi -r^2 {\rm Cs}^{2n} \phi {\rm Sn}\phi}{\left(1- {\rm Cs}^{n+1} \phi {\rm Sn}\phi \right) + r{\rm Cs}^{2n+1} \phi}\\
\label{dr/d_phi}
	&\equiv \frac{\gamma_1(\phi) r - \gamma_2(\phi)r^2 }{1- {\rm Cs}^{n+1} \phi {\rm Sn}\phi } + f(\phi, r)
\end{align}
with
\begin{equation*}
\gamma_1(\phi) = {\rm Sn}^2\phi {\rm Cs}^n \phi,\quad 
\gamma_2(\phi) = {\rm Sn}\phi {\rm Cs}^{2n} \phi + \frac{\gamma_1(\phi) {\rm Cs}^{2n+1}\phi}{1 - {\rm Sn}\phi {\rm Cs}^{n+1} \phi}
\end{equation*}
and
\begin{equation}
\label{hot-Lienard}
f(\phi, r) = \frac{-\gamma_2(\phi) }{1- {\rm Cs}^{n+1} \phi {\rm Sn}\phi } \sum_{l=1}^\infty \left( \frac{ -{\rm Cs}^{2n+1} \phi}{1- {\rm Cs}^{n+1} \phi {\rm Sn}\phi} \right)^l r^{l+2} = O(r^3)
\end{equation}
We have used the fact that, for sufficiently small $r$, we have the following series expression of the denominator in (\ref{dr/d_phi}) : 
\begin{equation*}
\frac{1}{\left(1- {\rm Cs}^{n+1} \phi {\rm Sn}\phi \right) + r{\rm Cs}^{2n+1} \phi} = 
\frac{1}{1- {\rm Cs}^{n+1} \phi {\rm Sn}\phi } \sum_{l=0}^\infty \left( \frac{ -{\rm Cs}^{2n+1} \phi}{1- {\rm Cs}^{n+1} \phi {\rm Sn}\phi} \right)^l r^l
\end{equation*}
for expressing $\gamma_1(\phi)$ and $\gamma_2(\phi)$.
\par
Following the argument in \cite{DH1999}, we seek the solution $r=r(\phi;r_0)$ of the form
\begin{equation}
\label{sol-r}
r(\phi;r_0) = \beta_1(\phi)r_0 + \beta_2(\phi)r_0^2 + O(r_0^3),\quad r(0;r_0) = r_0
\end{equation}
and the Poincar\'{e} map is defined in a small neighborhood of $r=0$ and is given by $\bar P(r_0) \equiv r(T;r_0)$.
\par
Differentiating (\ref{sol-r}) with respect to $\phi$ and comparing with (\ref{dr/d_phi}), $\beta_1$ and $\beta_2$ are turned out to be solutions of the following differential equations:
\begin{equation}
\begin{cases}
\beta_1'(\phi) = \gamma_1(\phi)\beta_1(\phi) / \left\{ 1 - {\rm Sn}\phi {\rm Cs}^{n+1}\phi \right\} , &\\
\beta_2'(\phi) = \left\{ \gamma_1(\phi)\beta_2(\phi) - \gamma_2(\phi) \beta_1^2(\phi)\right\} / \left\{1 - {\rm Sn}\phi {\rm Cs}^{n+1}\phi \right\}, & \\
\beta_1(0) = 1,\quad \beta_2(0) = 0. &
\end{cases}
\end{equation}
The exact forms of $\beta_1$ and $\beta_2$ are
\begin{equation*}
\begin{cases}
\beta_1(\phi) = e^{\alpha(\phi)}, &\\
\beta_2(\phi) = \displaystyle{
- e^{\alpha(\phi)} \int_0^\phi \frac{e^{\alpha(\psi)}\gamma_2(\psi) d\psi}{1 - {\rm Sn}\psi {\rm Cs}^{n+1}\psi}
} &\\
\end{cases}
\end{equation*}
with
\begin{equation*}
\alpha(\phi) = \int_0^\phi \frac{\gamma_1(\psi) d\psi}{1 - {\rm Sn}\psi {\rm Cs}^{n+1}\psi}.
\end{equation*}
Now we study the stability of the periodic orbit at infinity $\{r=0\}$, which follows from the exact form of $P(r_0)$.
First observe that both $\gamma_1(\phi)$ and $\gamma_2(\phi)$ are $T$-periodic and
\begin{align*}
&\gamma_1(-\phi) =  \gamma_1(\phi),\quad \gamma_1(T/2-\phi) = -\gamma_1(\phi),\quad \gamma_1(T/2+\phi) = -\gamma_1(\phi),\\
&\gamma_2(-\phi) = -{\rm Sn}\phi {\rm Cs}^{2n} \phi + \frac{\gamma_1(\phi) {\rm Cs}^{2n+1}\phi}{1 + {\rm Sn}\phi {\rm Cs}^{n+1} \phi},\\
&\gamma_2(T/2 - \phi) = {\rm Sn}\phi {\rm Cs}^{2n} \phi + \frac{\gamma_1(\phi) {\rm Cs}^{2n+1}\phi}{1 - {\rm Sn}\phi {\rm Cs}^{n+1} \phi} = \gamma_2(\phi),\\
&\gamma_2(T/2 + \phi) = -{\rm Sn}\phi {\rm Cs}^{2n} \phi + \frac{\gamma_1(\phi) {\rm Cs}^{2n+1}\phi}{1 + {\rm Sn}\phi {\rm Cs}^{n+1} \phi} = \gamma_2(-\phi)
\end{align*}
hold. Note that $n$ is now assumed to be odd.
Using these facts, we have
\begin{align*}
\alpha(T) &= -\int_0^T \cdots d\psi = -\left( \int_0^{T/4} +  \int_{T/4}^{T/2} +  \int_{T/2}^{3T/4} +  \int_{3T/4}^T \right)\cdots d\psi,\\
\int_{3T/4}^T \frac{\gamma_1(\psi) d\psi}{1 - {\rm Sn}\psi {\rm Cs}^{n+1}\psi} &=  \int_{0}^{T/4} \frac{\gamma_1(\mu) d\mu}{1 + {\rm Sn}\mu {\rm Cs}^{n+1}\mu}\quad \text{(via $\psi = -\mu + T$)},\\
\int_{T/4}^{T/2} \frac{\gamma_1(\psi) d\psi}{1 - {\rm Sn}\psi {\rm Cs}^{n+1}\psi} &=  -\int_{0}^{T/4} \frac{\gamma_1(\mu) d\mu}{1 - {\rm Sn}\mu {\rm Cs}^{n+1}\mu}\quad \text{(via $\psi = -\mu + T/2$)},\\
\int_{T/2}^{3T/4} \frac{\gamma_1(\psi) d\psi}{1 - {\rm Sn}\psi {\rm Cs}^{n+1}\psi} &=  -\int_{0}^{T/4} \frac{\gamma_1(\mu) d\mu}{1 + {\rm Sn}\mu {\rm Cs}^{n+1}\mu}\quad \text{(via $\psi = \mu + T/2$)}.
\end{align*}
Summarizing these equalities, we have $\alpha(T) = 0$.
In particular, we have $\beta_1(T) = 1$ and hence the periodic orbit at infinity is non-hyperbolic.
\par
\bigskip
Similarly, we calculate $\beta_2(T)$.
Letting 
\begin{align}
\label{Gamma_2}
\Gamma_2(\phi) &= \gamma_2(\phi) / (1-{\rm Sn}\phi {\rm Cs}^{n+1}\phi) \equiv \Gamma_{21}(\phi) + \Gamma_{22}(\phi),\\
\notag
\Gamma_{21}(\phi) &= \frac{ {\rm Sn}\phi {\rm Cs}^{2n} \phi } {1 - {\rm Sn}\phi {\rm Cs}^{n+1} \phi},\quad  
\Gamma_{22}(\phi) = \frac{\gamma_1(\phi) {\rm Cs}^{2n+1}\phi}{(1 - {\rm Sn}\phi {\rm Cs}^{n+1} \phi)^2},
\end{align}
we study the behavior of $\Gamma_2$.
Since $\beta_2(\phi) = -e^{\alpha(\phi)}\int_0^\phi e^{\alpha(\psi)}\Gamma_2(\psi) d\psi$, the behavior of $\beta_2(\phi)$ is dominated by that of $\Gamma_2$.
Instead we calculate
\begin{equation*}
\alpha_{2j}(T) = \int_0^T \Gamma_{2j}(\phi)d\phi, \quad j=1,2.
\end{equation*}
First, we have
\begin{align*}
\alpha_{21}(T) &= \int_0^T \Gamma_{21}(\phi)d\phi = \left( \int_0^{T/4} + \int_{T/4}^{T/2} + \int_{T/2}^{3T/4} + \int_{3T/4}^T \right)\Gamma_{21}(\phi)d\phi,\\
\int_{3T/4}^T \Gamma_{21}(\phi)d\phi &= -\int_0^{T/4} \frac{ {\rm Sn}\mu {\rm Cs}^{2n} \mu } {1 + {\rm Sn}\mu {\rm Cs}^{n+1} \mu}d\mu\quad \text{(via $\phi = -\mu + T$)},\\
\int_{T/4}^{T/2} \Gamma_{21}(\phi)d\phi &= \int_0^{T/4} \frac{ {\rm Sn}\mu {\rm Cs}^{2n} \mu } {1 - {\rm Sn}\mu {\rm Cs}^{n+1} \mu}d\mu\quad \text{(via $\phi = -\mu + T/2$)},\\
\int_{T/2}^{3T/4} \Gamma_{21}(\phi)d\phi &= -\int_0^{T/4} \frac{ {\rm Sn}\mu {\rm Cs}^{2n} \mu } {1 + {\rm Sn}\mu {\rm Cs}^{n+1} \mu}d\mu \quad \text{(via $\phi = \mu + T/2$)}.
\end{align*}
Therefore we have
\begin{equation*}
\alpha_{21}(T) = 2\int_0^{T/4}{\rm Sn}\mu {\rm Cs}^{2n} \mu \left( \frac{1} {1 - {\rm Sn}\mu {\rm Cs}^{n+1} \mu} - \frac{1} {1 + {\rm Sn}\mu {\rm Cs}^{n+1} \mu} \right)d\mu > 0.
\end{equation*}
In particular, the integrand is always negative for $\mu \in (0,T/4)$, since ${\rm Sn}\mu$ and ${\rm Cs}\mu$ are always positive and ${\rm Sn}\mu{\rm Cs}^{n+1}\mu < 1$.
\par
Similarly, we have
\begin{align*}
\alpha_{22}(T) &= \int_0^T \Gamma_{22}(\phi)d\phi = \left( \int_0^{T/4} + \int_{T/4}^{T/2} + \int_{T/2}^{3T/4} + \int_{3T/4}^T \right)\Gamma_{22}(\phi)d\phi,\\
\int_0^{T/4} \Gamma_{22}(\phi)d\phi &= \int_0^{T/4} \frac{ \gamma_1(\mu) {\rm Cs}^{2n+1} \mu } {(1 - {\rm Sn}\mu {\rm Cs}^{n+1} \mu)^2} d\mu,\\
\int_{3T/4}^T \Gamma_{22}(\phi)d\phi &= \int_0^{T/4} \frac{ \gamma_1(\mu) {\rm Cs}^{2n+1} \mu } {(1 + {\rm Sn}\mu {\rm Cs}^{n+1} \mu )^2}d\mu \quad \text{(via $\phi = -\mu + T$)},\\
\int_{T/4}^{T/2} \Gamma_{22}(\phi)d\phi &= \int_0^{T/4} \frac{ \gamma_1(\mu) {\rm Cs}^{2n+1} \mu } {(1 - {\rm Sn}\mu {\rm Cs}^{n+1}\mu)^2}d\mu \quad \text{(via $\phi = -\mu + T/2$)},\\
\int_{T/2}^{3T/4} \Gamma_{22}(\phi)d\phi &= \int_0^{T/4} \frac{ \gamma_1(\mu) {\rm Cs}^{2n+1} \mu } {(1 + {\rm Sn}\mu {\rm Cs}^{n+1} \mu)^2}d\mu \quad \text{(via $\phi = \mu + T/2$)}.
\end{align*}
Therefore we have
\begin{equation*}
\alpha_{22}(T) = 2\int_0^{T/4}{\rm Sn}\phi {\rm Cs}^{2n} \mu \left( \frac{1} {(1 - {\rm Sn}\mu {\rm Cs}^{n+1} \mu)^2} + \frac{1} {(1 + {\rm Sn}\mu {\rm Cs}^{n+1} \mu)^2} \right)d\mu > 0.
\end{equation*}
In particular, the integrand is always positive for $\mu \in (0,T/4)$.
\par
Summarizing the arguments, we have
\begin{align*}
\beta_2(T) &= \displaystyle{
- \int_0^T \frac{e^{\alpha(\psi)}\gamma_2(\psi) d\psi}{1 - {\rm Sn}\psi {\rm Cs}^{n+1}\psi}
} \\
	&=  -\left( \int_0^{T/4}+ \int_{T/4}^{T/2}+ \int_{T/2}^{3T/4} + \int_{3T/4}^T\right) e^{\alpha(\psi)}(\Gamma_{21}(\psi) + \Gamma_{22}(\psi))d\psi < 0
\end{align*}
since $e^{\alpha(\psi)}$ is always positive.
This inequality indicates that, for sufficiently small $r_0 > 0$, 
\begin{equation*}
0 < \beta_1(T)r_0 + \beta_2(T)r_0^2 + O(r_0^3) = (\beta_1(T) + \beta_2(T)r_0)r_0 + O(r_0^3) < r_0,
\end{equation*}
which shows that the periodic trajectory $\{r=0\}$ is attracting.

\subsection{Proof of Proposition \ref{prop-per-blowup-r2R}}
\label{proof-per-blowup-r2R}

Our approach here is to define a {\em smooth and locally positive} transform
\begin{equation}
\label{transform}
R = h(r,\phi) \equiv a_1(\phi) r + a_2(\phi) r^2
\end{equation}
such that $R$ is dominated by a vector field of the form (\ref{simpler-R}).
Substituting (\ref{transform}) into (\ref{dr/d_phi}), we have
\begin{align*}
\frac{dR}{d\phi} &= \frac{\partial h}{\partial r}(r,\phi)\frac{dr}{d\phi} + \frac{\partial R}{\partial \phi}\\
	&= \left( a_1(\phi) + 2a_2(\phi)r \right) \cdot \left\{ \frac{\gamma_1(\phi) r - \gamma_2(\phi)r^2 }{1- {\rm Cs}^{n+1} \phi {\rm Sn}\phi  } + f(\phi, r) \right\} + \left(\frac{d a_1}{d \phi}(\phi) r + \frac{d a_2}{d \phi}(\phi) r^2\right)
\end{align*}
with $f(\phi, r)$ in (\ref{hot-Lienard}). 
If the transform $h$ yield the ansatz (\ref{simpler-R}) for $R=h(r,\phi)$, the following equations must be satisfied:
\begin{align*}
&\frac{\gamma_1(\phi) a_1(\phi)}{1-{\rm Cs}^{n+1}\phi {\rm Sn}\phi} + \frac{d a_1}{d \phi}(\phi) = 0,\\
&\frac{2\gamma_1(\phi) a_2(\phi) - \gamma_2(\phi)a_1(\phi)}{1-{\rm Cs}^{n+1}\phi {\rm Sn}\phi} + \frac{d a_2}{d \phi}(\phi) = -\frac{ \gamma_2(\phi) a_1(\phi)^2 }{1- {\rm Cs}^{n+1} \phi {\rm Sn}\phi}
\end{align*}
with $a_1(0) = 1, a_2(0) = 0$.
The function $a_1(\phi)$ is explicitly written by
\begin{equation*}
a_1(\phi) = e^{\tilde \alpha(\phi)},\quad \tilde \alpha(\phi) = -\int_0^\phi \frac{\gamma_1(\psi) d\psi}{1 - {\rm Sn}\psi {\rm Cs}^{n+1}\psi},
\end{equation*}
which indicates that $a_1(\phi)$ is positive.
Similarly $a_2$ is calculated as
\begin{align*}
a_2(\phi) &= e^{2\tilde \alpha(\phi)} \int_0^\phi \frac{(e^{-\tilde \alpha(\psi)}-1)\gamma_2(\psi) d\psi}{1 - {\rm Sn}\psi {\rm Cs}^{n+1}\psi}.
%
\end{align*}

Note that
\begin{equation*}
\tilde \alpha(\theta) < 0 \text{ if }\theta \in (0,T/2),\quad \tilde \alpha(\theta) > 0 \text{ if }\theta \in (T/2, T),
\end{equation*}
equivalently
\begin{equation*}
(e^{-\tilde \alpha(\psi)}-1)< 0 \text{ if }\theta \in (0,T/2),\quad (e^{-\tilde \alpha(\psi)}-1)> 0 \text{ if }\theta \in (T/2, T).
\end{equation*}
We then consider the property of $a_2(\phi)$.
It is represented as
\begin{equation*}
a_2(\phi) = e^{2\tilde \alpha(\phi)} \int_0^\phi \frac{(e^{-\tilde \alpha(\psi)}-1) d\psi}{1 - {\rm Sn}\psi {\rm Cs}^{n+1}\psi}\cdot \left( {\rm Sn}\psi {\rm Cs}^{2n} \psi + \frac{\gamma_1(\psi) {\rm Cs}^{2n+1}\psi}{1 - {\rm Sn}\psi {\rm Cs}^{n+1} \psi} \right).
\end{equation*}
The denominator $1 - {\rm Sn}\psi {\rm Cs}^{n+1} \psi$ is always positive.
Furthermore,
\begin{align*}
\left( {\rm Sn}\psi {\rm Cs}^{2n} \psi + \frac{\gamma_1(\psi) {\rm Cs}^{2n+1}\psi}{1 - {\rm Sn}\psi {\rm Cs}^{n+1} \psi} \right) 
&=  \frac{{\rm Sn}\psi {\rm Cs}^{2n} \psi (1 - {\rm Sn}\psi {\rm Cs}^{n+1} \psi ) + {\rm Sn}^2\psi {\rm Cs}^n \psi  {\rm Cs}^{2n+1}\psi}{1 - {\rm Sn}\psi {\rm Cs}^{n+1} \psi} \\
&=  \frac{{\rm Sn}\psi {\rm Cs}^{2n} \psi }{1 - {\rm Sn}\psi {\rm Cs}^{n+1} \psi}.
\end{align*}

The function ${\rm Sn}\psi {\rm Cs}^{2n}\psi$ is positive for $\psi\in (0,T/2)$ and negative for $\psi\in (T/2, T)$, which are completely equal to the sign of $(e^{-\tilde \alpha(\phi)}-1)$.
\par
\bigskip
This concludes that
\begin{equation*}
\int_0^\phi \frac{(e^{-\tilde \alpha(\psi)}-1) {\rm Sn}\psi {\rm Cs}^{2n} \psi  d\psi}{ (1 - {\rm Sn}\psi {\rm Cs}^{n+1}\psi )^2} \geq 0, \quad \forall \phi \in [0,T).
\end{equation*}
We thus conclude that $a_2(\phi)$ is always nonnegative for $\psi \geq 0$.
If we define $R = h(r,\phi)$ as
\begin{align}
\notag
h(r,\phi) &\equiv e^{\tilde \alpha(\phi)}r + \left( e^{2\tilde \alpha(\phi)} \int_0^\phi \frac{(e^{-\tilde \alpha(\psi)}-1)\gamma_2(\psi) d\psi}{1 - {\rm Sn}\psi {\rm Cs}^{n+1}\psi} \right) r^2,\\
\tilde \alpha(\phi) &= -\int_0^\phi \frac{\gamma_1(\psi) d\psi}{1 - {\rm Sn}\psi {\rm Cs}^{n+1}\psi},
\end{align}
the vector field (\ref{dr/d_phi}) is transformed into
\begin{equation*}
\frac{dr}{d\phi} =\frac{- \gamma_2(\phi)R^2 }{1- {\rm Cs}^{n+1} \phi {\rm Sn}\phi } + O(R^3).
\end{equation*}
%

Since the exponential function is always positive, and since $\tilde \alpha_i$ are analytic in $\phi$,
then it gives a positive and smooth equivalence between (\ref{dr/d_phi}) and (\ref{simpler-R}).

\subsection{Proof of Proposition \ref{prop-asym-R}}
\label{proof-asym-R}
First note that both $R^{\rm original}(\phi)$ and $R^{\rm pri}(\phi)$ converge to $0$ as $\phi\to \infty$ for sufficiently small initial data $R_0 > 0$.
Then the L'H$\hat {\rm o}$pital's rule shows
\begin{align*}
-1 &= \lim_{\phi\to \infty} \frac{(dR^{\rm original}/d\phi)}{\frac{\gamma_2(\phi)(R^{\rm original})^2 }{1- {\rm Cs}^{n+1} \phi {\rm Sn}\phi } } 
= \lim_{\phi\to \infty} \frac{dR^{\rm original}}{(R^{\rm original})^2} 
\left(\frac{d\phi}{ \frac{\gamma_2(\phi)}{1- {\rm Cs}^{n+1} \phi {\rm Sn}\phi} } \right)^{-1}\\
&= \lim_{\phi \to \infty} \int_{R_0}^{R^{\rm original}(\phi)} r^{-2}dr \left(\int_0^\phi \Gamma_2(\psi) d\psi\right)^{-1}.
\end{align*}
Since
\begin{equation*}
-\int_0^\phi \Gamma_2(\psi) d\psi = \int_{R_0}^{R^{\rm pri}(\phi)} r^{-2} dr \equiv (R^{\rm pri})^{-1}(R^{\rm pri}(\phi); R_0),
\end{equation*}
we have
\begin{equation*}
(R^{\rm pri})^{-1}(R^{\rm original}(\phi); R_0) = \int_0^{\phi} \Gamma_2(\psi) d\psi\quad \text{ as }\phi\to \infty,
\end{equation*}
equivalently
\begin{align*}
R^{\rm original}(\phi) &= R^{\rm pri}(\int_0^{\phi} \Gamma_2(\psi) d\psi (1+o(1)))\\
	&= R^{\rm pri}(C_3\phi(1+o(1)))\text{ as }\phi\to \infty,
\end{align*}
where we have used the asymptotic behavior $\int_0^{\phi} \Gamma_2(\psi) d\psi \sim C_3 \phi$ as $\phi\to \infty$ stated in Lemma \ref{lem-Gamma2}.



\end{document}